\documentclass[aos,seceqn,citesort,dvips]{arximspdf}
\usepackage{dcolumn}
\usepackage{graphicx}

\doi{10.1214/10-AOS844}
\volume{39}
\issue{1}
\pubyear{2011}
\firstpage{556}
\lastpage{583}

\makeatletter

\newcolumntype{d}[1]{D{.}{.}{#1}}

\newcommand{\Prr}{\mathbf{P}}
\newcommand{\trp}{^{t}}

\newproclaim{defn}{Definition}[section]
\newtheorem{theo}{Theorem}[section]
\newtheorem{prop}{Proposition}[section]

\newproclaim{Condition}{Condition}

\makeatother

\begin{document}
\begin{frontmatter}

\title{Power-enhanced multiple decision functions controlling
family-wise error and false~discovery~rates}
\runtitle{Improved multiple testing procedures}

\begin{aug}
\author[A]{\fnms{Edsel A.} \snm{Pe\~na}\corref{}\thanksref{t1}\ead[label=e1]{pena@stat.sc.edu}
\ead[label=u1,url]{http://www.stat.sc.edu/\textasciitilde pena}},
\author[B]{\fnms{Joshua D.} \snm{Habiger}\ead[label=e2]{jhabige@okstate.edu}} and
\author[A]{\fnms{Wensong} \snm{Wu}\ead[label=e3]{wu26@mailbox.sc.edu}}
\runauthor{E. Pe\~na, J. D. Habiger, W. Wu}
\affiliation{University of South Carolina, Columbia,
Oklahoma State University\break and University of South Carolina, Columbia}
\address[A]{E. Pe\~na\\
W. Wu\\
Department of Statistics\\
University of South Carolina\\
Columbia, South Carolina 29208\\
USA\\
\printead{e1}\\
\phantom{E-mail: }\printead*{e3}\\
\printead{u1}} %adresu isvedimo komanda gale!
\address[B]{J. D. Habiger\\
Department of Statistics\\
Oklahoma State University\\
301-G MSCS Bldg\\
Stillwater, Oklahoma 74078\\
USA\\
\printead{e2}}
\end{aug}

\thankstext{t1}{Supported by NSF Grant DMS-08-05809,
NIH Grant RR17698 and
EPA Grant RD-83241902-0 to University
of Arizona with
subaward number Y481344 to the University of South Carolina.}

% HISTORY:
\received{\smonth{3} \syear{2010}}
\revised{\smonth{7} \syear{2010}}

% ABSTRACT
%
\begin{abstract}
Improved procedures, in terms of smaller missed discovery rates (MDR),
for performing multiple hypotheses testing
with weak and strong control of the family-wise error rate (FWER)
or the false discovery rate (FDR) are developed and studied. The
improvement over existing procedures such as the {\v{S}}id{\'a}k
procedure for
FWER control and the Benjamini--Hochberg (BH) procedure for FDR control
is achieved by exploiting possible differences in the powers of the
individual tests. Results signal the need to take into account the
powers of the individual tests and to have multiple hypotheses decision
functions which are not limited to simply using the individual
$p$-values, as is the case, for example, with the {\v{S}}id{\'a}k,
Bonferroni, or
BH procedures. They also enhance understanding of the role of the
powers of individual tests, or more precisely the receiver operating
characteristic (ROC) functions of decision processes, in the search for
better multiple hypotheses testing procedures. A decision-theoretic
framework is utilized, and through auxiliary randomizers the procedures
could be used with discrete or mixed-type data or with rank-based
nonparametric tests. This is in contrast to existing $p$-value based
procedures whose theoretical validity is contingent on each of these
$p$-value statistics being stochastically equal to or greater than a
standard uniform variable under the null hypothesis. Proposed
procedures are relevant in the analysis of high-dimensional ``large
$M$, small $n$'' data sets arising in the natural, physical, medical,
economic and social sciences, whose generation and creation is
accelerated by advances in high-throughput technology, notably, but not
limited to, microarray technology.
\end{abstract}

% KEYWORDS
%
\begin{keyword}[class=AMS]
\kwd[Primary ]{62F03}
\kwd[; secondary ]{62J15}.
\end{keyword}
\begin{keyword}
\kwd{Benjamini--Hochberg procedure}
\kwd{Bonferroni procedure}
\kwd{decision pro\-cess}
\kwd{false discovery rate (FDR)}
\kwd{family wise error rate (FWER)}
\kwd{Lagrangian optimization}
\kwd{Neyman--Pearson most powerful test}
\kwd{microarray analysis}
\kwd{reverse martingale}
\kwd{missed discovery rate (MDR)}
\kwd{multiple decision function and process}
\kwd{multiple hypotheses testing}
\kwd{optional sampling theorem}
\kwd{power function}
\kwd{randomized $p$-values}
\kwd{generalized multiple decision $p$-values}
\kwd{ROC function}
\kwd{{\v{S}}id{\'a}k procedure}.
\end{keyword}

\end{frontmatter}

%s1 ###
\section{Introduction and motivation}
\label{section-Introduction}

The advent of modern technology, epitomized by the microarray, has led
to the generation of very high-dimensional data pertaining to
characteristics of a large number, $M$, of attributes, hereon called
genes, associated with usually a small number, $n$, of units or
subjects. Several such data sets are, for example, described in
\cite{Efr08}, and these are the inputs to so-called parallel inference
problems. The most common form of inference is multiple hypotheses
testing, wherein for the $m$th gene there are two competing hypotheses,
a null hypothesis $H_{m0}$ and an alternative hypothesis $H_{m1}$, for
which a decision is to be made based on the data. In such multiple
decision-making, there is a need to be cognizant and cautious of the
\textit{Hyde}-ian nature of multiplicity, while also exploiting the
\textit{Jekyll}-ian potentials of multiplicity \cite{Stevenson03}.
Furthermore, this entails a tenuous balance between two competing
desires: controlling the rate of rejection of correct null hypotheses,
while at the same time maintaining the rate of discovery of correct
alternative hypotheses.

As in single-pair hypothesis testing, a type I error occurs
when a correct null hypothesis is rejected, while a type II
error occurs when a false null hypothesis is not rejected. Several
type I errors have been proposed in multiple testing; see
\cite{DudShaBol03} and \cite{DudLaa08}.
Our focus is on the weak family wise error rate
(FWER), the probability of rejecting at least one null
hypothesis when all the nulls are correct; strong FWER, the
probability of rejecting at least one correct null hypothesis; and
false discovery rate (FDR), the expected
proportion of the number of false rejections of nulls
relative to the number of rejections \cite{Sor89,BenHoc95}.
Our type II error rate is the missed discovery rate
(MDR), the expected number of false nonrejections of null hypotheses.
Other type II errors have been discussed in
\cite{DudShaBol03,Sto03,DudGilVan07,Efr07,DudLaa08}.
The usual framework in developing multiple decision functions is to
bound the chosen type~I error rate, and then minimize or make small the
MDR. For example, a procedure controlling weak FWER, under an
independence assumption, is that of {\v{S}}id{\'a}k \cite{Sid67}; while
a conservative one not requiring independence is the Bonferroni
procedure~\cite{Bon37}. For FDR control, the most common procedure is
the BH procedure \cite{BenHoc95}. Control of type~I error measures
related to the FDR have also been discussed in
\cite{EfrTibStoTus01,GenWas02,Sto02,Sto03,Efr04,Efr07,SunCai07,Efr08},
while \cite{SchSpj82,LanLinFer05,JinCai07} focused on estimation of the
proportion of correct null hypotheses.

Procedures like the {\v{S}}id{\'a}k, Bonferroni and BH, rely on the
set of $p$-values of individual tests. Their validity hinges on each
$p$-value statistic being stochastically equal to or greater than a
standard uniform variable under the null hypothesis. This fails,
however, with noncontinuous variables or when rank-based nonparametric
tests are used. Crucially, $p$-value based procedures also do not
exploit the power characteristics of the individual tests, contrary to
Neyman and Pearson's \cite{NeyPea33} adage that such considerations
are germane in constructing optimal tests. Such $p$-value based
procedures are fine in exchangeable settings where power
characteristics of the individual tests are identical, but not in
situations where genes or subclasses of genes have different
structures; see \cite{EfrAAS08,FerFriRueThoKon08,RoqWie09}.

Some papers dealing with procedures exploiting the power functions are
\mbox{\cite{Spj72,WesKriYou98}}. The use of weighted
$p$-values to improve type II performance have also been explored in
\cite{GenRoeWas06,WasRoe06,RubDudLaa06,KanYeLiuAllGao09,RoqWie09}.
Other approaches for optimal procedures are those in
\cite{Sto07,StoDaiLee07} which employ a Neyman--Pearson approach
and \cite{SunCai07} where oracle and adaptive compound rules were
obtained. Compound
rules are characterized by information borrowing from each of the genes,
so a decision function for a specific gene utilizes information
from other genes. Decision-theoretic
and Bayesian approaches were also implemented in
\cite{MulParRobRou04,SarZhoGho08,ScoBer06,Efr08,GuiMulZha09}.
More recently, \cite{EfrAAS08} argues for separate subclass analysis, while
\cite{FerFriRueThoKon08} proposed use of external covariates, with the
procedures having a Bayes and empirical Bayes flavor.

The main goal of this paper is to develop better multiple testing
procedures controlling weak FWER, strong FWER and FDR by taking into
account the individual powers of the tests. We focus on the most
fundamental setting where the null and alternative hypotheses for each
gene are both simple. This is also the setting in~\cite{RoqWie09}. This
admits, as starting point, the Neyman--Pearson most powerful (MP) test
for each pair of hypotheses. Each MP test will have a power, but we
will see that it is beneficial to look at each of these powers as
function of their MP test's size, their so-called receiver operating
characteristic (ROC) function.

The paper proceeds as follows.
Section \ref{section-Mathematical Setting} presents the
decision-theoretic elements.
Section \ref{section-Revisiting MP Tests} reviews and reexamines
MP tests, $p$-value statistics and ROC functions.
Section \ref{section-Optimal FWER Control} develops the optimal weak
FWER-controlling procedure, with existence and uniqueness established
in Section \ref{section-Existence}. Section \ref{subsection-Finding
Optimal Sizes} analytically describes the procedure for differentiable
ROC functions.
Section \ref{subsection-Concrete Examples} provides a concrete
example using normal distributions, while Section
\ref{subsection-Investing} discusses a size-investing strategy for
optimality.
Section \ref{section-Restrictions} discusses limitations, extensions
and connections:
Section \ref{subsection-Restricting to D0} deals with the
restriction to the class of simple procedures;
Section \ref{subsection-MLR Families} deals with extensions to
the composite hypotheses setting in the presence of the monotone
likelihood ratio (MLR) property; and
Section \ref{subsection-In Terms of P-Values} relates the
optimal procedure to weighted $p$-value based procedures.
Section \ref{section-strong FWER control} develops an improved
procedure which
strongly controls the FWER, whereas
Section \ref{section-FDR procedures} develops an
improved procedure which controls FDR.
The development of
these new procedures is anchored on the weak FWER-controlling optimal
procedure. We establish that the
sequential {\v{S}}id{\'a}k and BH
procedures are special cases of these more general procedures.
Section \ref{section-Simulated Comparison} provides
a modest simulation study demonstrating that the new FDR-controlling procedure
improves on the BH procedure.
Section \ref{section-Concluding Remarks} contains a summary and
some concluding remarks.

%%%%
To manage the length of the paper and provide more focus on the
main ideas and results, technical proofs of lemmas, propositions,
theorems and
corollaries are all gathered in the supplemental article \cite{PenHabWu10Supp}.
%%%%

%s2 ###
\section{Mathematical setting}
\label{section-Mathematical Setting}

Let $(\Omega,\mathcal{F},\Prr)$ be a probability space
and $\mathcal{M} = \{1,2,\ldots,M\}$ an index set with $M$ a known positive
integer. For each $m \in\mathcal{M}$, let $X_m\dvtx(\Omega,\mathcal
{F}) \rightarrow(\mathcal{X}_m,
\mathcal{B}_m)$, $\mathcal{X}_m$ some space with $\sigma$-field of
subsets $\mathcal{B}_m$. Form the product space $(\mathcal
{X},\mathcal{B})$ with
$\mathcal{X} = \times_{m \in\mathcal{M}} \mathcal{X}_m$ and
$\mathcal{B} = \sigma(
\times_{m\in\mathcal{M}}\mathcal{B}_m)$ so
$X = (X_1,X_2,\ldots,X_M)\dvtx
(\Omega,\mathcal{F}) \rightarrow(\mathcal{X},\mathcal{B})$.
The probability measure
of $X$ is $Q = \Prr X^{-1}$, while the (marginal) probability measure of
$X_m$ is $Q_m =
PX_m^{-1}$.
For each $m \in\mathcal{M}$, let $Q_{m0}$ and $Q_{m1}$ be two known
probability measures on
$(\mathcal{X}_m,\mathcal{B}_m)$. We assume that $Q \in\mathcal{Q}$,
a class
of probability measures on $(\mathcal{X},\mathcal{B})$ with marginal
probability measure $Q_m \in\{Q_{m0},Q_{m1}\}$
for each $m \in\mathcal{M}$. Let
$\theta= (\theta_1,\ldots,\theta_M)\dvtx\mathcal{Q} \rightarrow
\Theta\equiv\{0,1\}^M$
with $\theta_m(Q) = I\{Q_m = Q_{m1}\}$,
$I\{\cdot\}$ denoting indicator function.
Define, for each $Q \in\mathcal{Q}$, the subcollections
%
%%\begin{eqnarray}
$\mathcal{M}_0 \equiv\mathcal{M}_0(Q) = \{m \in\mathcal{M}\dvtx\theta
_m(Q) = 0\}$
%%\label{M0 subindices} \\
and
$\mathcal{M}_1 \equiv\mathcal{M}_1(Q) = \{m \in\mathcal{M}\dvtx\theta
_m(Q) = 1\}$.
%%\label{M1 subindices}
%%\end{eqnarray}
%
In this paper, we shall impose an \textit{independence condition} given by:
\renewcommand{\theCondition}{(I)}
\begin{Condition}\label{condI}
$(X_m, m \in\mathcal{M}_0(Q))$ is an independent collection of random
entities,
that is, $\forall B_m \in\mathcal{B}_m$,
%
%%\begin{equation}
%%\label{independence under nulls}
$Q(\times_{m \in\mathcal{M}_0(Q)} B_m) = \prod_{m \in\mathcal{M}_0(Q)}
Q_m(B_m)$.
%%\end{equation}
\end{Condition}

However, the collection $(X_m, m \in\mathcal{M}_1(Q))$ need not be an
independent collection, but it is independent of
$(X_m, m \in\mathcal{M}_0(Q))$.
Two extreme subcollections of $\mathcal{Q}$ are
%
%%\begin{eqnarray}
%%\label{extreme subcollections}
$\mathcal{Q}_0 = \{Q \in\mathcal{Q}\dvtx\theta_m(Q) = 0, \forall
m \in\mathcal{M}\}$
%%\label{extremal subcollection 0} \\
%%
and
$\mathcal{Q}_1 = \{Q \in\mathcal{Q}\dvtx\theta_m(Q) = 1, \forall
m \in\mathcal{M}\}$.
%%\label{extremal subcollection 1}
%%\end{eqnarray}
%
By Condition \ref{condI}, $\mathcal{Q}_0$ is a singleton
set, $Q_0$ will denote its element; while
$\mathcal{Q}_1$ need not be a singleton set.
The decision problem is
to determine $\mathcal{M}_0(Q)$ and $\mathcal{M}_1(Q)$ based on
$X$, which is equivalent to simultaneously testing
the $M$ pairs of hypotheses
$H_{m0}\dvtx Q_m = Q_{m0}$ versus
$H_{m1}\dvtx Q_m = Q_{m1}$ for $m \in\mathcal{M}$.

We adopt a decision-theoretic framework
similar to \cite{SarZhoGho08}.
The \textit{action space}
is $\mathcal{A} = \{0,1\}^M$ with generic element $a = (a_1,a_2,\ldots
,a_M)\trp\in\mathcal{A}$
with $a_m = 0 (1)$ meaning $H_{m0}$ is accepted (rejected).
The \textit{parameter space} is $\mathcal{Q}$, though the effective
parameter space is $\Theta= \{0,1\}^M$ with generic element $\theta=
(\theta_1,
\theta_2,\ldots,\break \theta_M)\trp$. We introduce several \textit{loss
functions}, $L\dvtx\mathcal{A} \times\mathcal{Q} \rightarrow\Re_+$,
defined via
%
%e2.3 ###
%e2.2 ###
%e2.1 ###
%
\begin{eqnarray}
\label{k-FWER loss}
L_{0}(a,Q) & = & I\bigl\{a\trp\bigl(1-\theta(Q)\bigr) \ge1\bigr\};
\\
\label{FDR loss}
L_1(a,Q) & = & \biggl[ \frac{a\trp(1-\theta(Q))}{a\trp1}\biggr] I\{
a\trp1 > 0\};
\\
\label{MDR loss}
L_{2}(a,Q) & = & {(1-a)\trp\theta(Q)},
\end{eqnarray}
with the convention that $0/0 = 0$ and $1$ is an $M \times1$ vector of $1$'s.
The loss function
$L_{0}(a,Q)$ equals 1 if and only if at least one false discovery is
committed.
The loss $L_{1}(a,Q)$ is the \textit{false discovery proportion},
being the ratio between the number of false discoveries and the
number of discoveries; whereas the loss $L_2(a,Q)$ is
the \textit{number of missed discoveries} being the number of true
alternative hypotheses that were not discovered. We focus on this missed
discovery number since the relevant question is how many correct
alternatives [$\theta(Q)\trp1$] were missed by using the action $a$?
See also \cite{RoqWie09} which essentially uses this
loss function to induce their power metric.
Other types of losses, such as the false negative proportion with
%
%%\begin{displaymath}
$(a,Q) \mapsto{[(1-a)\trp\theta(Q)]}/{[(1-a)\trp1]} I\{(1-a)\trp1
> 0\}$,
%%\end{displaymath}
%
have also been considered; see \cite{GenWas02,SarZhoGho08}.

A \textit{nonrandomized} multiple decision function (MDF) is
a $\delta\dvtx(\mathcal{X},\mathcal{B}) \rightarrow(\mathcal
{A}$, $\sigma(\mathcal{A}))$, where
$\sigma(\mathcal{A})$ is the power set of $\mathcal{A}$. Such an MDF
may be represented by
%
%%\begin{displaymath}
$\delta(x) = (\delta_1(x),\delta_2(x),\ldots,\delta_M(x))\trp$,
%%\end{displaymath}
%
where $\delta_m(x) \in\{0,1\}$. In general, each $\delta_m$ could
be made to depend on the
full data $x$ instead of just $x_m$. We denote by
$\mathcal{D}$ the class of all nonrandomized
MDFs.
A \textit{randomized} MDF may also be considered. Denote by $\mathcal
{P}(\mathcal{A})$
the space of all probability measures over $(\mathcal{A},\sigma
(\mathcal{A}))$.
A randomized MDF
is a $\delta^*\dvtx (\mathcal{X},\mathcal{B}) \rightarrow(\mathcal
{P}(\mathcal{A}),\sigma(\mathcal{P}(\mathcal{A})))$. For a realization
$X = x$, an action is chosen from $\mathcal{A}$
according to the probability measure $\delta^*(x)$.
Denote by $\mathcal{D}^*$ the space of all randomized MDFs. Clearly,
$\mathcal{D}
\subset\mathcal{D}^*$.
By augmenting data $X$ with a randomizer $U \sim U(0,1)$ which is
independent of $X$, randomized MDFs could be made nonrandomized with
respect to the \textit{augmented data} $(X,U)$.
Henceforth, $\mathcal{D}$ represents all \textit{non}randomized MDFs
$\delta(X,U)$'s based on $(X,U)$.

For brevity of notation, $\Prr_Q\{f(X,U) \in B\}$ and $E_Q\{f(X,U)\}$
represent probability and expectation with respect to $(X,U)$ with $X
\sim Q$, $U \sim U(0,1)$ and $X$ and $U$ independent. For
$\delta\in\mathcal{D}$ and the loss functions defined earlier, we have
the \textit{risk functions}
%
%e2.4 ###
%
\begin{eqnarray}
\label{risk 0k}
R_{0}(\delta,Q) &=& E_Q \{L_{0}(\delta(X,U),Q)\};
\\
%
%e2.5 ###
\label{risk fdr}
R_1(\delta,Q) &=& E_Q\{
L_1(\delta(X,U),Q)\};
%
%e2.6 ###
\\
\label{risk mdr}
R_2(\delta,Q) &=& E_Q \{
L_2(\delta(X,U),Q)\}.
\end{eqnarray}

Given a $\delta= (\delta_1,\delta_2,\ldots,\delta_M)\trp$, let
%
%%\begin{displaymath}
$\pi_\delta(Q) = (\pi_{\delta_1}(Q),\pi_{\delta_2}(Q),\ldots,\pi
_{\delta_M}(Q))\trp$
%%\end{displaymath}
%
with
%
%%\begin{eqnarray}
$\pi_{\delta_m}(Q) = E_Q \{\delta_m(X,U)\}$
%%\label{power of decision function}
%%\end{eqnarray}
be its vector of power functions.
Then (\ref{risk mdr}) becomes
%
%%\begin{equation}
%%\label{mdr risk in terms of powers}
$R_2(\delta,Q) =
(1 - \pi_\delta(Q))\trp\theta(Q)$.
%%\end{equation}
%
In terms of these risk functions, for $\delta\in\mathcal{D}$,
its weak FWER is
%
%%\begin{displaymath}
$\operatorname{FWER}(\delta) = R_{0}(\delta,Q_0)$.
%%\end{displaymath}
%
If each $\delta_m$ depends only on $X_m$ and $U$, by Condition \ref{condI},
%
%e2.7 ###
%
\begin{equation}
\label{weak FWER}
\operatorname{FWER}(\delta) = 1 - E\biggl\{\prod_{m \in\mathcal{M}}
[1 - \Prr_{Q_{m0}}\{\delta_m(X_m,U) = 1|U\}]\biggr\},
\end{equation}
where the expectation is with respect to $U$.
When $Q = Q_0$ and with the $m$th component $\delta_m^*$ of the
randomized MDF depending only on $X_m$,
an alternative formulation is
to have $U = (U_1, U_2, \ldots, U_M)$ a vector of i.i.d.
$U(0,1)$ variables which is independent of the $X_m$'s.
The $m$th component may then be
redefined via $\delta_m(X_m,U_m) = I\{U_m \le\delta_m^*(X_m)\}$.
Then (\ref{weak FWER}) becomes
%
%%\begin{equation}
%%\label{weak FWER 2}
$\operatorname{FWER}(\delta) = 1 - \prod_{m \in\mathcal{M}}
[1 - \Prr_{Q_{m0}}\{\delta_m(X_m,U_m) = 1\}]$.
%%\end{equation}
%

The risk function $R_1(\delta,Q)$ is the false discovery \textit{rate}
(FDR) of $\delta$ at $Q$ \cite{BenHoc95}; while the risk function
$R_2(\delta,Q)$ will be called the missed discovery \textit{rate} (MDR)
of $\delta$ at~$Q$. The adjective ``rate'' is somewhat misleading since
$R_2(\delta,Q)$ takes values in $[0,|\mathcal{M}_1(Q)|]$ instead of
$[0,1]$; however, this does not cause difficulty since, given the true
underlying probability measure $Q$ of $X$, $|\mathcal{M}_1(Q)|$ is
constant. This risk is related to the expected number of true positives
(ETP), an error measure used in \cite{Spj72,Sto07}, via
$\operatorname{ETP}(\delta,Q) = |\mathcal{M}_1(Q)| - R_2(\delta,Q)$.

To find an optimal MDF \textit{weakly} controlling FWER in a subclass
\mbox{$\mathcal{D}_0 \subseteq\mathcal{D}$}, a~threshold $\alpha\in(0,1)$
is specified and then we seek a $\delta^* \in\mathcal{D}_0$ with
%
%%\begin{displaymath}
$R_{0}(\delta^*,Q_0) = \operatorname{FWER}(\delta^*) \le\alpha$,
%%\end{displaymath}
%
and such that for any $\delta\in\mathcal{D}_0$ satisfying
$R_{0}(\delta,Q_0) = \break \operatorname{FWER}(\delta) \le\alpha$, we have
%
%%\begin{equation}
%%\label{minimax criterion}
$\sup_{Q \in\mathcal{Q}} R_2(\delta^*,Q) \le\sup_{Q \in\mathcal
{Q}} R_2(\delta,Q)$.
%%\end{equation}
%
This criterion has a minimax flavor. One may require
only that $R_2(\delta^*,Q^*) \le R_2(\delta,Q^*)$ where $Q^*$ is the
true, but unknown,
probability law of $X$; but this may be too strong to preclude a
solution to the optimization problem. However, see \cite{Sto07} for a
situation with
a different type I error and where an optimal, albeit an oracle,
solution for
minimizing $R_2(\delta,Q^*)$ is possible.
Observe that for $\delta\in\mathcal{D}$, by using the representation of
$R_2(\delta,Q)$ in terms of the powers,
%
%%\begin{equation}
%%\label{sup of mdr}
$\sup_{Q \in\mathcal{Q}} R_2(\delta,Q) = \sup_{Q \in\mathcal
{Q}_1} R_2(\delta,Q) =
M - \inf_{Q \in\mathcal{Q}_1} \sum_{m \in\mathcal{M}}
\pi_{\delta_m}(Q)$.
%%\end{equation}
%
The optimality condition on the MDR
amounts therefore to maximizing
$\sum_{m \in\mathcal{M}} \pi_{\delta_m}(Q_{m1})$.
Interestingly, if we had standardized the loss function $L_2(a,Q)$ to take
values in $[0,1]$ via division by $|\mathcal{M}_1(Q)| = \theta(Q)\trp1$,
the minimax justification does not carry through!

For \textit{strong} FWER control, we seek
a compound MDF, $\delta^* \in\mathcal{D}$, with
%
%%\begin{equation}
%%\label{strong FWER control condition}
$R_{0}(\delta^*$, $Q^*) \le\alpha$
%%\end{equation}
%
\textit{whatever} the true, but unknown, probability law $Q^*$ of $X$
is, and with $\sum_{m \in\mathcal{M}} \pi_{\delta_m^*}(Q_{m1})$ large,
possibly maximal, among all $\delta\in\mathcal{D}$ satisfying
$R_{0}(\delta,Q^*) \le\alpha$.
For (strong) FDR-control, a threshold $q^* \in(0,1)$ is specified
and we seek a compound MDF, $\delta^* \in\mathcal{D}$,
such that, \textit{whatever} $Q^*$ is,
%
%%\begin{equation}
%%\label{FDR condition}
$R_1(\delta^*,Q^*) \le q^*$,
%%\end{equation}
%
and with
$\sum_{m \in\mathcal{M}} \pi_{\delta_m^*}(Q_{m1})$ large, possibly
maximal, among
all $\delta\in\mathcal{D}$ satisfying $R_{1}(\delta,Q^*) \le q^*$.
For discussion of weak
and strong control, refer to \cite{DudShaBol03,DudLaa08}.
Discussion of optimality in multiple testing can be found in \cite{LehRomSha05}
where maximin optimality results are established for some step-down and
step-up MTPs.

%s3 ###
\section{Revisiting MP tests and $p$-value statistics}
\label{section-Revisiting MP Tests}

An MDF $\delta\in\mathcal{D}$ whose $m$th component $\delta_m$
depends only on
$(X_m,U_m)$ for every $m \in\mathcal{M}$ is called simple; otherwise,
it is
compound. The subclass of simple MDFs,
denoted by $\mathcal{D}_0$, will be our initial focus in searching for an
optimal weak FWER-controlling MDF. The resulting optimal MDF will
then anchor our search for strong FWER- and FDR-controlling compound MDFs.
Before implementing this program, we introduce the unifying concept of
decision processes.

%s3.1 ###
\subsection{Decision processes, ROC functions, $p$-value statistics}
\label{section-Most Powerful Tests}

First, a brief review. Let $X\dvtx (\Omega,\mathcal{A}) \rightarrow
(\mathcal{X},\mathcal{B})$ and $Q = \Prr X^{-1}$. Based on $X$,
consider testing the pair of hypotheses
%
%%\begin{equation}\label{simple hypotheses}
$H_0\dvtx Q = Q_0$ versus $H_1\dvtx Q = Q_1$,
%%\end{equation}
%
where $Q_0$ and $Q_1$ are two probability measures on $(\mathcal
{X},\mathcal{B})$. Let $q_0$
and $q_1$ be versions of the densities of $Q_0$ and $Q_1$ with respect
to some
fixed dominating measure $\nu$, for example, $\nu= Q_0 + Q_1$.
Recall that a test or decision function is a
%
%%\begin{displaymath}
$\delta\dvtx (\mathcal{X},\mathcal{B}) \rightarrow([0,1],\sigma[0,1])$,
%%\end{displaymath}
%
with $\sigma[0,1]$ the Borel sigma-field on $[0,1]$.
Given $X = x$, $\delta(x)$ is the
probability of deciding in favor of $H_1$. Its size is
%
%%\begin{equation}\label{size}
$\alpha_\delta= E_{Q_0}\delta(X);$
%%\end{equation}
%
it is of level $\alpha\in[0,1]$ if
$\alpha_\delta\le\alpha$. Its power is
%
%%\begin{equation}\label{power}
$\pi_\delta= E_{Q_1} \delta(X)$.
%%\end{equation}
%
$\delta^*$ is most powerful (MP)
of level $\alpha$ if $\alpha_{\delta^*}
\le\alpha$ and for all $\delta$ with $\alpha_\delta\le\alpha$,
we have
$\pi_{\delta^*} \ge\pi_{\delta}$.
\begin{defn}
\label{defn-decision process}
A collection $\Delta= \{\delta_\eta\dvtx \eta\in[0,1]\}$ of test
functions such that,
a.e. $[Q]$, $\delta_0(x) = 0$, $\delta_1(x) = 1$ and $\eta\mapsto
\delta_\eta(x)$
is nondecreasing and right-continuous, is a {decision} process.
Its \textit{size} function is
$A_\Delta\dvtx [0,1] \rightarrow[0,1]$ and its \textit{power} function is
$\rho_\Delta\dvtx [0,1]
\rightarrow[0,1]$, where
%
%%\begin{displaymath}
$A_\Delta(\eta) = \alpha_{\delta_\eta} = E_{Q_0} \delta_\eta(X)$ and
%%
%%\mbox{and}
%%
$\rho_\Delta(\eta) = \pi_{\delta_\eta} = E_{Q_1} \delta_\eta(X)$.
%%\end{displaymath}
%
Its \textit{receiver operating characteristic} (ROC) curve is
$\operatorname{ROC}(\Delta) \equiv\operatorname{Graph}\{(A_\Delta(\eta),
\rho_\Delta(\eta))\dvtx \eta\in[0,1]\}$.
If $A_\Delta(\eta) = \eta$ for all $\eta\in[0,1]$,
$\eta\mapsto\rho_\Delta(\eta)$ is the \textit{ROC function} of
$\Delta$.
\end{defn}

The use of the phrase \textit{power function} in
Definition \ref{defn-decision process} is atypical since we are not viewing
this as a function of a parameter as is the usual meaning of this phrase.
However, for lack of a better name, we shall adopt this terminology.
In the sequel, $\delta_\eta$ and $\delta(\eta)$
will be used interchangeably to also represent $\delta(\cdot;\eta)$.

Let $L\dvtx (\mathcal{X},\mathcal{B}) \rightarrow(\Re_+,\sigma(\Re
_+))$ be a version of the
likelihood ratio function:
$L(x) = q_1(x)/q_0(x)$ a.e. $[\nu]$.
Let $G_0(\cdot)$ and $G_1(\cdot)$ be the distribution
functions of $L(X)$ when $\mathcal{L}(X) = Q_0$ and $\mathcal{L}(X) =
Q_1$, where
$\mathcal{L}(X)$ is probability measure of $X$. For a monotone nondecreasing
right-continuous function $M(\cdot)$ from $\Re$ into $\Re$, let
%
%%\begin{displaymath}
$M^{-1}(r) = \inf\{x \in\Re\dvtx M(x) \ge r\}$
and
$\Delta M(r) = M(r) - M(r-)$.
%%\end{displaymath}
%
By the Neyman--Pearson fundamental lemma \cite{NeyPea33},
the MP test function of level $\eta$ for testing
$H_0$ versus $H_1$ is
%
%e3.1 ###
%
\begin{equation}
\label{NP Test}
\delta^*(X;\eta) \equiv\delta^*_\eta
= I\{L(X) > c(\eta)\} + \gamma(\eta) I\{L(X) = c(\eta)\},
\end{equation}
where
%
%%\begin{displaymath}
$c(\eta) = G_0^{-1}(1-\eta)$ and
$\gamma(\eta) = {(G_0(c(\eta)) - (1-\eta))}/{\Delta G_0(c(\eta))}$.
%%\end{displaymath}
%
Let $U \sim U(0,1)$ be independent of $X$. Redefine $\delta^*$ via
$\delta_\eta^{**} \equiv\delta^{**}(X,U;\eta) = I\{U \le\delta
^*(X;\eta)\}$,
%
%%\begin{displaymath}
%%\delta^{**}(X,U;\eta) = \delta^{**}_\eta= I\{\delta^*(X;\eta) = 1\}
%%+ I\{\delta^*(X;\eta) = \gamma(\eta); U \le\gamma(\eta)\}.
%%\end{displaymath}
%
which is nonrandomized w.r.t. $(X,U)$. In essence, with the aid of an auxiliary
randomizer $U$, the MP test could always be made nonrandomized.
The decision process formed from these MP tests, given by
%
%e3.2 ###
%
\begin{equation}
\label{MP test process}
\Delta^* = \{\delta_\eta^*\dvtx \eta\in[0,1]\} = \{\delta_\eta
^{**}\dvtx
\eta\in[0,1]\},
\end{equation}
is called the most powerful (MP) decision process.
The power (at $Q = Q_1$) of the MP test $\delta_\eta^*$ or $\delta
_\eta^{**}$ is
%
%e3.3 ###
%
\begin{equation}
\label{NP Test Power}
\rho_{\Delta^*}(\eta) \equiv\pi_{\delta^*_\eta} = \pi_{\delta
^{**}_\eta} =
1 - G_1(c(\eta)) + \gamma(\eta) \Delta G_1(c(\eta)).
\end{equation}
It is well known \cite{Leh97} that $\pi_{\delta^*_\eta} < 1$ implies
$\alpha_{\delta^*_\eta} = \eta$.
We denote by $A_{\Delta^*}$ and
$\rho_{\Delta^*}$ the size and power functions of $\Delta^*$.
If $\pi_{\delta^*_\eta} < 1$ for all $\eta< 1$, then
%
%%\begin{equation}
%%\label{mapping power}
$\eta\mapsto\rho_{\Delta^*}(\eta)$
%%\end{equation}
%
is the ROC function of $\Delta^*$. We present below some important
properties of this function.

Before stating the proposition, we reiterate that all formal
proofs of propositions, theorems, lemmas and corollaries are in
the supplemental article \cite{PenHabWu10Supp}.

\begin{prop}\label{prop-properties of power}
The function $\rho_{\Delta^*}\dvtx [0,1] \rightarrow[0,1]$ in
(\ref{NP Test Power}) is
concave, continuous and nondecreasing.
Furthermore, $\rho_{\Delta^*}(\eta) \ge\eta$ and it is
strictly increasing on the set
$\mathcal{N}_{<} \equiv\{\eta\in[0,1]\dvtx \rho_{\Delta^*}(\eta) <
1\}$.
\end{prop}
\begin{defn}
\label{defn-p value statistic}
Let $\Delta= \{\delta_\eta\dvtx \eta\in[0,1]\}$ be a decision process,
where $\delta_\eta\dvtx (\mathcal{X} \times[0,1], \mathcal{B} \otimes
\sigma[0,1])
\rightarrow(\{0,1\},\sigma\{0,1\})$. Its (randomized) $p$-value
statistic is
$S_\Delta\dvtx (\mathcal{X} \times[0,1], \mathcal{B} \otimes\sigma[0,1])
\rightarrow([0,1],\sigma[0,1])$ with
%
%%\begin{displaymath}
$S_\Delta(x,u) = \inf\{\eta\in[0,1]\dvtx \delta_\eta(x,u) = 1\}$.
%%\end{displaymath}
%
\end{defn}

When $\forall(\eta, x, u)\dvtx \delta_\eta(x,u) = \delta_\eta(x)$,
then $S_\Delta(X,U)$ is the usual
$p$-value statistic. See also
\cite{CoxHin74} for a more specialized definition of
a randomized $p$-value statistic.
We refer the reader to \cite{HabPen10} for properties of this $p$-value
statistic and its use in existing FDR-controlling procedures.
\begin{prop}\label{prop-distributions of p value}
Let $\Delta= \{\delta_\eta\dvtx \eta\in[0,1]\}$ be a decision process with
$p$-value statistic $S_\Delta$.
Then, for all $s \in[0,1]$, $H_0(s) \equiv\Prr_{Q_0}(S_\Delta\le s)
= A_\Delta(s)$
and $H_1(s) \equiv\Prr_{Q_1}(S_\Delta\le s) = \pi_{\delta(s)} =
\rho_{\Delta}(s)$.
Consequently, $S_\Delta\sim U[0,1]$ under $\mathcal{L}(X) = Q_0$ if
and only if
$\forall\eta\in[0,1]\dvtx A_\Delta(\eta) = \eta$.
\end{prop}

%s4 ###
\section{Optimal \textit{weak} FWER control}
\label{section-Optimal FWER Control}

Return now to the multiple decision problem in Section
\ref{section-Mathematical Setting}. We extend the notion of
decision processes to the multiple decision setting.
\begin{defn}
\label{defn-multiple decision process}
A collection $\bolds{\Delta} = (\Delta_m\dvtx m \in\mathcal{M})$,
where $\Delta_m =
(\delta_m(\eta)\dvtx \eta\in[0,1])$ is a decision process on $(\mathcal
{X} \times
[0,1]^M, \mathcal{B} \otimes\sigma[0,1]^M)$, is
a multiple decision process (MDP). It is {simple} if each $\Delta_m$
is simple; otherwise, it is {compound}.
When simple its {multiple decision size function} is
%
%%\begin{equation}\label{multiple decision size functions}
$\mathbf{A}_{\bolds{\Delta}} =
(A_{\Delta_m}\dvtx m \in\mathcal{M})$
%%\end{equation}
%
and its {multiple decision ROC function} is
%
%%\begin{equation}\label{multiple decision ROC functions}
$\bolds{\rho}_{\bolds{\Delta}} =
(\rho_{\Delta_m}\dvtx m \in\mathcal{M})$,
%%\end{equation}
%
where $A_{\Delta_m}$ and $\rho_{\Delta_m}$ are the size and ROC functions
of $\Delta_m$.
\end{defn}

%s4.1 ###
\subsection{Optimization problem}
\label{subsection-optimization}
Let $\bolds{\Delta}$ be a simple MDP. Then, a
multiple decision size vector $\bolds{\eta} = (\eta_m\dvtx m \in
\mathcal{M}) \in
\mathcal{N} \equiv[0,1]^M$ determines from $\bolds{\Delta}$ an MDF
%
%%\begin{displaymath}
$\delta_{\bolds{\Delta}}(\bolds{\eta}) =
(\delta_m(\eta_m)\dvtx m \in\mathcal{M}) \in\mathcal{D}_0$.
%%\end{displaymath}
%
For this MDF,
%
%%\begin{displaymath}
$\operatorname{FWER}(\delta_{\bolds{\Delta}}(\bolds{\eta})) =
1 - \prod_{m \in\mathcal{M}} [1 - A_{\Delta_m}(\eta_m)]$
%%\end{displaymath}
%
and
%
%%\begin{displaymath}
$R_2(\delta_{\bolds{\Delta}}(\bolds{\eta}),Q_1) =
M - \sum_{m \in\mathcal{M}} \rho_{\Delta_m}(\eta_m)$
%%\end{displaymath}
%
for $Q_1 \in\mathcal{Q}_1$.
Fix an FWER-threshold $\alpha\in(0,1)$. Suppose there exists
a multiple decision size vector $\bolds{\eta}_{\bolds{\Delta
}}^*(\alpha) \in\mathcal{N}$
such that
\[
\bolds{\eta}_{\bolds{\Delta}}^*(\alpha) =
\mathop{\arg\max}_{\bolds{\eta} \in\mathcal{N}}
\biggl\{\sum_{m \in\mathcal{M}} \rho_{\Delta_m}(\eta_m)\dvtx
\prod_{m \in\mathcal{M}} [1 - A_{\Delta_m}(\eta_m)]
\ge1 - \alpha\biggr\}.
\]
Then, $\mathbf{A}_{\Delta}(\bolds{\eta}_{\Delta}^*(\alpha)) =
(A_{\Delta_m}(\eta_{\bolds{\Delta},m}^*(\alpha))\dvtx m \in\mathcal
{M})$ is the optimal
multiple decision size vector for weak FWER control at $\alpha$ associated
with the simple MDP $\bolds{\Delta}$. The associated optimal simple
MDF is
$\delta_{\bolds{\Delta}}(\mathbf{A}_{\Delta}(\bolds\eta
_{\bolds{\Delta}}^*(\alpha)))$.

But, since $H_{m0}$ and $H_{m1}$ are both simple, then there exists a simple
most powerful MDP,
%
%%\begin{equation}
%%\label{simple MP MDP}
$\bolds{\Delta}^* = (\Delta_m^*\dvtx m \in\mathcal{M})$,
%%\end{equation}
%
where $\Delta_m^* = (\delta_m^*(\eta)\dvtx \eta\in[0,1])$ with
$\delta_m^*(\eta)$ being the simple Neyman--Pearson MP test function
of size $\eta$ for $H_{m0}$ versus $H_{m1}$. Consider the simple MDF
obtained from
$\bolds{\Delta}^*$ given by
%
%%\begin{displaymath}
$(
\delta_m^*(A_{\Delta_m}(\bolds{\eta}_{\bolds{\Delta
},m}^*(\alpha)))\dvtx m \in\mathcal{M}
)$.
%%\end{displaymath}
%
This will satisfy the FWER constraint, and by virtue of the MP
property of each $\delta_m^*(A_{\Delta_m}(\bolds{\eta}_{
\bolds\Delta,m}^*(\alpha)))$
for each $m \in\mathcal{M}$,
\[
\sum_{m \in\mathcal{M}}
\rho_{\Delta_m^*}(A_{\Delta_m}(\bolds{\eta}_{\bolds{\Delta
},m}^*(\alpha)))
\ge
\sum_{m \in\mathcal{M}}
\rho_{\Delta_m}(A_{\Delta_m}(\bolds{\eta}_{\bolds{\Delta
},m}^*(\alpha))).
\]
Thus, in searching for the optimal weak FWER-controlling
simple MDF, it suffices to restrict to the simple most powerful MDP
$\bolds{\Delta}^*$. Without loss of generality (wlog), we may assume
$A_{\Delta_m^*}(\eta) = \eta$ for $m \in\mathcal{M}$ and $\eta\in[0,1]$.
The optimization problem reduces to finding a
$\bolds{\eta}_{\bolds{\Delta}^*}^*(\alpha) \in\mathcal{N}$ satisfying
%
%e4.1 ###
%
\begin{equation}
\label{optimal FWER size vector}
\bolds{\eta}_{\bolds{\Delta}^*}^*(\alpha) =
\mathop{\arg\max}_{\bolds{\eta} \in\mathcal{N}}
\biggl\{\sum_{m \in\mathcal{M}} \rho_{\Delta_m^*}(\eta_m)\dvtx
\prod_{m \in\mathcal{M}} (1 - \eta_m)
\ge1 - \alpha\biggr\}.
\end{equation}
The optimal \textit{weak} FWER-controlling simple MDF is then
%
%e4.2 ###
%
\begin{equation}
\label{optimal weak FWER controlling MDF}
\delta^*_{W}(\alpha) \equiv
\bigl(\delta_m^*(\eta_{\bolds{\Delta}^*,m}^*(\alpha))\dvtx m \in\mathcal{M}\bigr).
\end{equation}

Two well-known and conventional choices for the size vector
$\bolds{\eta} = (\eta_m\dvtx m \in\mathcal{M})$ which satisfy the
weak FWER constraint are the {\v{S}}id{\'a}k sizes
%
%%\begin{equation}\label{Sidak sizes}
$\eta_m = \eta_m(\alpha) = 1 - (1-\alpha)^{{1}/{M}}$
%%\end{equation}
%
and the Bonferroni-adjusted sizes
%
%%\begin{equation}\label{Bonferroni sizes}
$\eta_m = \eta_m(\alpha) = {\alpha}/{M}$.
%%\end{equation}
%
The former requires the independence Condition \ref{condI} and is sharp,
the latter is conservative but does not require Condition \ref{condI}. Both
ignore possible differences in power traits of the individual
test functions.

%s4.2 ###
\subsection{Existence and uniqueness of optimal size vector}
\label{section-Existence}

We establish the existence of an optimal multiple decision size vector for
weak FWER control within the class $\mathcal{D}_0$.
As pointed out in Section \ref{subsection-optimization}, it
suffices to
look for the optimal weak FWER-controlling simple MDF by starting with the
most powerful simple MDP $\bolds{\Delta}^* = (\Delta_m^*\dvtx m \in
\mathcal{M})$.
For brevity, $\rho_m \equiv\rho_{\Delta_m^*}$ and
$A_m(\eta) \equiv A_{\Delta_m^*}(\eta) = \eta$. Recall that
$\mathcal{N} = [0,1]^M$, the multiple decision size space.
In a nutshell, the existence of an optimal multiple decision size vector
for weak FWER control exploits convexity properties of relevant subsets
of $\mathcal{N}$.
This is formalized by establishing a sequence of propositions which are
presented below.
For $\alpha\in[0,1]$, define the weak FWER constraint set
%
%e4.3 ###
%
\begin{equation}
\label{C alpha}
C_\alpha=
\cases{
\displaystyle \biggl\{\bolds{\eta} \in\mathcal{N}\dvtx\sum_{m \in\mathcal{M}}
\log(1-\eta_m) \ge\log(1-\alpha)\biggr\}, &\quad
if $\alpha< 1$, \cr
\mathcal{N}, &\quad if $\alpha= 1$.}
\end{equation}
\begin{prop}
\label{prop-properties of C alpha}
$C_\alpha$ satisfies \textup{(i)} $\bolds{\eta} = \mathbf{0}
\in C_\alpha$; \textup{(ii)} $(\mathbf{0},\alpha_m) \in C_\alpha$ for all \mbox{$m
\in\mathcal{M}$}, where
$(\mathbf{0},\alpha_m)$ is the zero-vector with the $m$th element
replaced by $\alpha$; and
\textup{(iii)} it is convex and closed.
\end{prop}
%
%%\begin{proof}
%%See Proof of Proposition \ref{prop-properties of C alpha} in
%%\end{proof}
%
\begin{prop}
\label{prop-boundary set}
For $\bolds{\eta}_0 \in\mathcal{N}$ let
$U(\bolds{\eta}_0) = \{ \bolds{\eta} \in\mathcal{N}\dvtx
\eta_m \ge\eta_{0m}, \forall m \in\mathcal{M}\}$,
the upper set of $\eta_0$, and let
$UB(C_\alpha) = \{\bolds{\eta} \in\mathcal{N}\dvtx C_\alpha\cap
U(\bolds{\eta}) =
\{\bolds{\eta}\}\}$,
the upper boundary set of $C_\alpha$.
Then, for all $\alpha\in[0,1)$,
%%\begin{displaymath}
$UB(C_\alpha) = \{
\bolds{\eta} \in\mathcal{N}\dvtx\sum_{m \in\mathcal{M}} \log
(1-\eta_m) = \log(1-\alpha)\}$.
%%\end{displaymath}
%
\end{prop}
%
%%\begin{proof}
%%See Proof of Proposition \ref{prop-boundary set} in
%%\end{proof}
%
\begin{prop}
\label{prop-properties of Nb}
Let $\mathcal{N}_b \equiv\{ \bolds{\eta} \in\mathcal {N}\dvtx\sum_{m
\in\mathcal{M}} \rho_m(\eta_m) \ge Mb\}$ for $b \in[0,1]$. Then
$\{\mathcal{N}_b\dvtx b \in[0,1]\}$ satisfies \textup{(i)}
$\bolds{\eta} = \mathbf{1} \in\mathcal{N}_b$, \textup{(ii)} it is
closed and convex, and \textup{(iii)} $\mathcal{N} = \mathcal{N}_0
\supseteq\mathcal{N}_{b_1} \supseteq\mathcal{N}_{b_2}$ for $0 \le b_1
\le b_2 \le1$.
\end{prop}
%
%%\begin{proof}
%%See Proof of Proposition 4.3 in \cite{PenHabWu10Supp}.
%%\end{proof}
%
\begin{prop}
\label{prop-characterization of B alpha}
Let $B_\alpha= \{ b \in[0,1]\dvtx \mathcal{N}_b \cap C_\alpha\ne
\varnothing\}$ for $\alpha\in[0,1)$
and let $b_\alpha^* = \sup B_\alpha$. Then
$B_\alpha= [0, b_\alpha^*]$.
\end{prop}

%%\begin{proof}
%%See Proof of Proposition \ref{prop-characterization of B alpha} in
%%\end{proof}

Building on these intermediate results, the
existence of an optimal weak FWER-controlling multiple decision size vector
is obtained.
\begin{theo}[(Existence)]
\label{theo-existence} Let $\alpha\in[0,1)$. Then $C_\alpha\cap\mathcal
{N}_{b_\alpha^*} \ne \varnothing$. Furthermore, $\eta\in\mathcal{N}$ is
a weak \mbox{FWER-}$\alpha$ optimal multiple decision size vector if
and only if $\eta\in C_\alpha\cap \mathcal{N}_{b_\alpha^*}$.
\end{theo}

%%\begin{proof}
%%See Proof of Theorem \ref{theo-existence} in \cite{PenHabWu10Supp}.
%%\end{proof}

%%\subsection{Uniqueness of Optimal Size Vector}
%%\label{section-Uniqueness}

Theorem \ref{theo-existence} guarantees existence of an optimal weak FWER
multiple decision size vector,
but it does not address whether the solution is unique. We present a result
on this issue in the following theorem.
\begin{theo}[(Uniqueness)]
\label{theo-uniqueness}
Let $\alpha\in[0,1)$ and define
$C_\alpha(m) = \{\eta_m \in[0,1]\dvtx \bolds{\eta} \in C_\alpha\}$,
called the $m$th section of $C_\alpha$.
If, for all $m \in\mathcal{M}$, the mapping $\eta_m \mapsto\rho
_m(\eta_m)$
is strictly increasing on $C_\alpha(m)$,
then the optimal weak \mbox{FWER-}$\alpha$ multiple decision
size vector is unique and it is the $\bolds{\eta}^*$ satisfying
$C_\alpha\cap\mathcal{N}_{b_\alpha^*} = \{\bolds{\eta}^*\}$.
\end{theo}

%%\begin{proof}
%%See Proof of Theorem \ref{theo-uniqueness} in \cite{PenHabWu10Supp}.
%%\end{proof}

It is easy to see that a sufficient condition for uniqueness
of the optimal size vector is that, for all $m \in\mathcal{M}$,
$\eta_m \in[0, \sup C_\alpha(\eta_m)) \Rightarrow\rho_m(\eta_m)
< 1$.
Nonuniqueness may occur with
nonregular families of densities, for example, uniform or shifted exponential,
where the power of the MP test may equal one even though its size is
still less than
one. It occurs if the decision processes in the MDP do \textit{not} satisfy
the condition that $\forall\eta\in[0,1], \forall m \in\mathcal{M},
A_m(\eta) = \eta$, which is the case with discrete data
or when using nonparametric rank-based test functions
with randomization not permitted.

%s4.3 ###
\subsection{Finding optimal size vector}
\label{subsection-Finding Optimal Sizes}

Generally, without differentiability of the ROC functions as in the
case with discrete distributions, linear or nonlinear programming
methods are needed to obtain the optimal solution. In the case,
however, where the ROC functions are twice-differentiable, the optimal
size vector is in a more explicit form.
\begin{theo}\label{theo-optimal sizes}
Let $\bolds{\Delta}^* = ( \Delta_m^*, m \in\mathcal{M})$
be the MP MDP, and assume that
the ROC functions $\eta_m \mapsto\rho_m(\eta_m)$ are strictly increasing
and twice-differentiable
with first and second derivatives $\rho_m^\prime$ and $\rho
_m^{\prime\prime}$,
respectively. Given $\alpha\in(0,1)$,
the optimal weak \mbox{FWER-}$\alpha$ multiple decision size vector
$\bolds{\eta}^* \equiv\bolds{\eta}_{\bolds{\Delta
}^*}^*(\alpha) = (\eta_m^*(\alpha),
m \in\mathcal{M})$ is the $\bolds{\eta} \in\mathcal{N}$ satisfying
\textup{(i)} for some $\lambda\in\Re_+$, $\forall m \in\mathcal{M},
\rho^\prime_m(\eta_m)(1-\eta_m) = \lambda$ and
\textup{(ii)} $\sum_{m \in\mathcal{M}} \log(1-\eta_m) = \log(1-\alpha)$.
%
%%\begin{eqnarray}
%%& \forall m \in\mathcal{M},\ {\rho_m^\prime(\eta_m)}(1-\eta_m) =
%%\mbox{for some}\ \lambda\in\Re_+; & \label{first cond} \\
%%& \sum_{m \in\mathcal{M}} \log(1-\eta_m) = \log(1-\alpha). &
%%\end{eqnarray}
%
\end{theo}

%%\begin{proof}
%%See Proof of Theorem \ref{theo-optimal sizes} in
%%\end{proof}

A question arises as to whether the optimal sizes are monotonic in
$\alpha$.
Such a property is desirable since it will imply that if at
FWER size $\alpha_1$ we have $\delta_m(\eta_m(\alpha_1)) = 1$,
then at an FWER size $\alpha_2$ with $\alpha_2 > \alpha_1$,
we will also have $\delta_m(\eta_m(\alpha_2)) = 1$.
This property will also be critical in proving a martingale
property needed for the development of the FDR-controlling procedure.
This issue is the content of the following proposition.
\begin{prop}
\label{prop-monotonicity of optimal solution}
Assume the conditions of Theorem \ref{theo-optimal sizes}. Then, for each
$m \in\mathcal{M}$, the mapping $\alpha\mapsto\eta_m^*(\alpha)$ is
nondecreasing and continuous.
\end{prop}

%s4.4 ###
\subsection{Gaussian example for weak FWER control}
\label{subsection-Concrete Examples}

For $m \in{\mathcal{M}}$, let $X_m \sim N(\mu_m,\sigma_{m0}^2)$,
where the
$\mu_m$'s are unknown and $\sigma_{m0}^2$'s are known. Consider the
multiple hypotheses
testing problem
%
%%\begin{equation}\label{normal hypotheses}
$H_{m0}\dvtx \mu_m = \mu_{m0}$ and $H_{m1}\dvtx \mu_m = \mu_{m1}$
%%\end{equation}
%
with $\mu_{m0} < \mu_{m1}$ for $m \in\mathcal{M}$. The MP
test of size $\eta_m$ for $H_{m0}$ versus $H_{m1}$ is\vspace*{1pt}
%
%%\begin{equation}
%%\label{MP test normal}
$\delta_m^{*}(X_m;\eta_m) \equiv\delta_m^{*}(\eta_m) =
I\{X_m \ge\mu_{m0} + \sigma_{m0} \Phi^{-1}(1 - \eta_m)\}$,
%%\end{equation}
%
where $\Phi(\cdot)$ and $\Phi^{-1}(\cdot)$ are the cumulative distribution
and quantile functions, respectively, of a standard normal variable.
The $m$th effect size is
%
%%\begin{equation}\label{effect size normal}
$\gamma_m = {(\mu_{m1}-\mu_{m0})}/{\sigma_{m0}}$,
%%\end{equation}
%
and the ROC function of the decision process
$\Delta_m^* = (\delta_m^{*}(\eta_m)\dvtx \eta_m \in[0,1])$
is
%
%%\begin{equation}\label{power function normal}
$\rho_m(\eta_m) \equiv\rho_m(\eta_m;\gamma_m) = \Phi(\gamma_m -
\Phi^{-1}(1-\eta_m))$,
%%\end{equation}
%
clearly twice-differentiable with respect to $\eta_m$.
With $\phi(\cdot)$ %% = \exp(-z^2)/\sqrt{2\pi} = \Phi^\prime(z)$
the standard normal density function,
%%the derivative of $\rho_m(\cdot)$ is
%
%%\begin{equation}\label{derivative of power normal}
%
\[
(\rho_m)^\prime(\eta_m) =
\frac{\phi(\gamma_m - \Phi^{-1}(1-\eta_m))}{\phi(\Phi
^{-1}(1-\eta_m))}.
\]
%
%%\end{equation}
%
For fixed $\alpha\in(0,1)$ and $\gamma_m$'s,
consider the mappings $d \mapsto\eta_m(d),
m \in\mathcal{M}$, defined
implicitly by the equation
%
%e4.4 ###
%
\begin{equation}\label{condition 1 normal}
\frac{\phi(\gamma_m - \Phi^{-1}(1-\eta_m))}{\phi(\Phi
^{-1}(1-\eta_m))}(1-\eta_m) - d = 0.
\end{equation}
The optimal value of $d$, denoted by $d^*$, solves the equation
%
%e4.5 ###
%
\begin{equation}\label{condition 2 normal}
\sum_{m \in\mathcal{M}} \log\bigl(1-\eta_m(d)\bigr) - \log(1-\alpha) = 0.
\end{equation}
The optimal sizes of the $M$ MP tests are then $\eta_m(d^*),
m \in\mathcal{M}$. An \texttt{R} \cite{IhaGen96}
implementation of this numerical problem first defines
$v_m = 1 - \Phi^{-1}(1-\eta_m)$, so condition (\ref{condition 1
normal}) amounts to
solving for $v_m = v_m(d)$ the equation
%
%e4.6 ###
%
\begin{equation}\label{equation in vm}
\log\Phi(v_m) + \gamma_mv_m - \log(d) - \gamma_m^2/2 = 0.
\end{equation}
We utilized a Newton--Raphson iteration in solving for $v_m$'s in (\ref{equation
in vm}) and the
\texttt{uniroot} routine in the \texttt{R} Library to solve for $d$ in (\ref
{condition 2 normal}).
Upon obtaining $v_m(d)$'s, the $\eta_m(d)$'s are computed via
$\eta_m(d) = 1 - \Phi(v_m(d))$.

Figure \ref{figure-normal test and uniform effect} demonstrates the optimal
sizes when $M = 2\mbox{,}000$ and for uniformly distributed effect sizes.
Observe from the second panel that when the effect size is small, which
converts to low
power, then the optimal size for the test is also small, but also
note that when the effect size is large, which converts to high
power, then the optimal test size is also small. For the tests with
moderate effect sizes or power, then the optimal sizes are higher.
This behavior could also be seen by looking at the third panel in
the figure which shows the achieved power of the tests at the
optimal sizes.

%f1 ###
%
\begin{figure}

\includegraphics{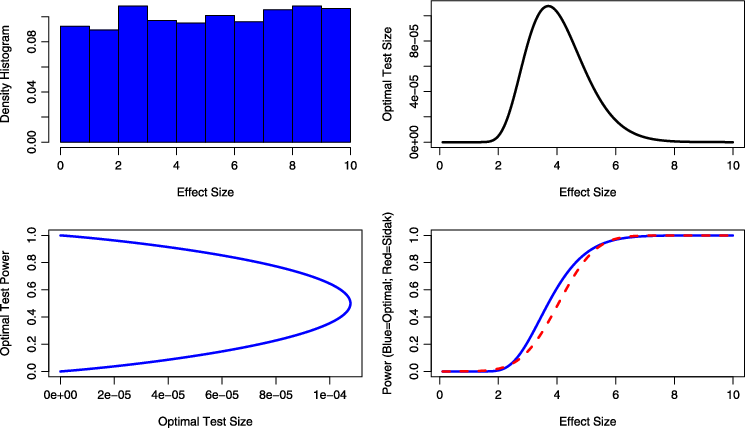}

\caption{Optimal test sizes and powers for $2\mbox{,}000$ MP tests of hypotheses
under normality when the effect sizes
were generated from a $\operatorname{uniform}[0.1,10]$ distribution.
Panel four shows the
powers for both the optimal [solid black] and the {\v{S}}id{\'a}k
[dashed red]
tests with respect to effect sizes.}
\label{figure-normal test and uniform effect}
\end{figure}

The efficiency of the optimal procedure
relative to the {\v{S}}id{\'a}k procedure was measured via the
ratio (multiplied by 100) of the average power over the $M$ tests,
defined by
$\sum_{m\in\mathcal{M}} \rho_m(\eta_m)/M$, of
the optimal procedure and the average power of the {\v{S}}id{\'a}k procedure.
The fourth panel in Figure \ref{figure-normal test and uniform
effect} depicts the
powers of the resulting tests versus the effect size for both
procedures (solid blue${}={}$optimal; dashed red${}={}${\v{S}}id{\'a}k). For
these uniformly-generated
effect sizes, the efficiency of the optimal procedure over the {\v
{S}}id{\'a}k
is 103.5\%. This efficiency is affected by the vector of effect sizes. For
instance, when we change the effect sizes in Figure
\ref{figure-normal test and uniform effect} to be generated from a
uniform over $[0.1,2]$, then the efficiency jumps to 181.7\%, though
it should also be pointed out that since the effect sizes are small,
then the overall powers of both procedures are also small.

%s4.5 ###
\subsection{A size-investing strategy}
\label{subsection-Investing}

In the preceding Gaussian example, as well as in other situations
we examined, for example, with exponential and Bernoulli distributions,
we observed the phenomenon where, among
the $M$ tests, those with low powers (small effect sizes) and those
with high powers
(large effect sizes) are allocated
relatively small sizes in the weak FWER-controlling optimal procedure.
The tests
with larger sizes are those with moderate powers or effect
sizes. This is a \textit{size-investing strategy} in the multiple hypotheses
testing problem, and it has intuitive content.
%%The theoretical basis for this strategy, at least under the
%%conditions of Theorem \ref{theo-optimal sizes}, is the first
%condition for
%%optimality in (\ref{first cond}), which
%%is tied to the rates of change of the ROC functions
%%of the MP multiple decision process, together with a penalty
%incurring from larger sizes.
%
%%This strategy can be explained intuitively.
With the overall goal of making more real discoveries while controlling
the proportion of false discoveries for a pre-specified, usually small,
overall size~$\alpha$, the optimal procedure dictates that not much
size should be accorded those tests with either very low or very high
powers. The former case will not lead to any discoveries anyway if the
size that could be allocated is small, while the latter case will lead
to discoveries even if the test sizes are made small. Thus, there is
more to be gained by investing larger sizes on those tests that are of
moderate powers, and an appropriate tweaking of their test sizes
according to condition (i) in Theorem~\ref{theo-optimal sizes} improves
the ability to achieve more real discoveries. However, this phenomenon
is dependent on the magnitude of the overall size. If this overall size
is made larger, more leeway ensues to the extent that it may then be
more beneficial to allocate more size to those with low powers since
those tests with moderate powers, when they had small sizes, may now
have larger powers because of the consequent increase in their sizes.
The precise and crucial determinant of where the differential sizes
should be allocated are the rates of change of the ROC functions, with
some size-attenuation. Interesting discussions of size and weight
allocation strategies can also be found in \cite{WesKriYou98}, where
the size allocation was related to the ``$\alpha$-spending'' function
of \cite{LandeM83}, in \cite{FosSti08} which deals with
$\alpha$-investing in sequential procedures that control expected false
discoveries, and in \cite{GenRoeWas06,RoqWie09} which discuss optimal
weights for the $p$-values.

A tangential
real-life manifestation of this strategy occurred during the 2008
American presidential election, with the total resources (financial,
manpower, etc.)
available to the
candidates analogous to the overall size in the multiple testing problem.
In the waning days of the campaign, the major
candidates, then-Senator Barack Obama of the
Democratic Party and Senator John McCain
of the Republican Party,
focused their campaign efforts, in terms of allocating their financial
and manpower resources,
in the ``battleground states'' of North Carolina, Virginia and
Pennsylvania, while
basically ignoring the ``in-the-bag states''
of South Carolina, then expected to vote for McCain, and California,
then expected to vote for Obama. Also, by
virtue of the deep resources of the Obama campaign, it was able to
allocate more
resources \textit{even} in states that traditionally voted Republican,
whereas the
McCain campaign, with a relatively smaller war chest, had to ``drop''
some states
(e.g., Michigan) in their campaign. The behaviors of the two camps
somehow mirror the size-investing strategy with proper accounting of each
campaign's overall resources.

%s5 ###
\section{Restrictions, extensions and connections}
\label{section-Restrictions}
%
%s5.1 ###
\subsection{On the restriction to $\mathcal{D}_0$}
\label{subsection-Restricting to D0}

The optimization problem for weak FWER control could be construed as
limited since we restricted to the subclass $\mathcal{D}_0$ thus
leading to an optimal weak FWER-controlling procedure that is still
simple. In \mbox{\cite{Sto07,SunCai07}}, it was demonstrated that
performance is enhanced via compound MDFs. Examples of compound MDFs
are the \textit{estimated} optimal discovery procedure (ODP) in
\cite{Sto07,StoDaiLee07}, the FDR-controlling procedure in
\cite{BenHoc95}, and the oracle-based adaptive MDFs in \cite{SunCai07}.

Could we immediately start from
compound MDFs in the search for an
optimal weak FWER-controlling compound
MDF? Let us suppose that
$\bolds{\delta} = (\delta_m\dvtx m \in\mathcal{M})$ is a compound MDF,
so $\delta_m$ depends on $({X},{U})$ and not only
on $(X_m,U_m)$. For such an MDF, we have
%
%e5.1 ###
%
\begin{equation}\label{RHS1}
R_{0}(\bolds{\delta},Q) =
\Prr_{Q}\biggl\{\bigcup_{m \in\mathcal{M}_0(Q)} [
\delta_m({X},{U}) = 1 ] \biggr\}.
%%\\
%%& = &
%%1 - \Prr_{Q}\{
%%\cap_{m \in\mathcal{M}_0(Q)} [
%%\delta_m({X},{U}) = 0 ]\}.
%%\label{RHS2}
\end{equation}
Now, even if the independence Condition \ref{condI} holds,
$(\delta_m({X},{U})\dvtx m \in\mathcal{M}_0(Q))$
need not be an independent collection. As such
no closed-form \textit{exact} expression for
$R_{0}(\bolds{\delta},Q)$ need exist.
The right-hand side in (\ref{RHS1}) could be Bonferroni-bounded by
%
%e5.2 ###
%
\begin{equation}
\label{EFP}
\operatorname{EFP}(\bolds{\delta},Q) \equiv\sum_{m \in\mathcal{M}_0(Q)}
\alpha_{\delta_m}(Q),
\end{equation}
called the expected number of false positives
in \cite{Sto07}. Alternatively,
if a generalized positive quadrant dependence (PQD)
condition holds, with
\[
\Prr_{Q}\biggl\{\bigcap_{m \in\mathcal{M}_0(Q)} [
\delta_m({X},{U}) = 0
]
\biggr\}
\ge
\prod_{m \in\mathcal{M}_0(Q)}\Prr_{Q}\{
\delta_m({X},{U}) = 0
\},
\]
then the right-hand side in (\ref{RHS1}) could be upper-bounded by
%
%e5.3 ###
%
\begin{equation}
\label{PQD upper bound}
\operatorname{PQD}(\bolds{\delta},Q) \equiv
1 - \prod_{m \in\mathcal{M}_0(Q)} [ 1 - \alpha_{\delta_m}(Q) ],
\end{equation}
where
$\alpha_{\delta_m}(Q) = E_{Q} \delta_m({X},{U})$,
the size of $\delta_m$ when $m \in\mathcal{M}_0(Q)$.
For this compound MDF, its MDR is
%
%%\begin{equation}
%%\label{MDR for compound MDF}
$R_2(\bolds{\delta},Q) = {\sum_{m \in\mathcal{M}_1(Q)} [
1 - \pi_{\delta_m}(Q) ]}$,
%%\end{equation}
%
where $\pi_{\delta_m}(Q) = E_{Q} \delta_m({X},{U})$ is the power
of $\delta_m$ when $m \in\mathcal{M}_1(Q)$.

An optimization approach could proceed by putting an upper
threshold $\alpha\in(0,1)$
on either (\ref{EFP}) or (\ref{PQD upper bound}), and then finding the
$\bolds{\delta}$ that minimizes $R_2(\bolds{\delta},Q)$,
or equivalently, maximizes
%
%%\begin{displaymath}
$\operatorname{ETP}(\bolds{\delta},Q) \equiv
\sum_{m \in\mathcal{M}_1(Q)} \pi_{\delta_m}(Q)$,
%%\end{displaymath}
%
the latter quantity referred to as the
expected number of true positives in \cite{Sto07}.
The MDFs in \cite{Spj72} and \cite{Sto07} were both obtained
through this program.
The MDF in \cite{Spj72} is
%
%e5.4 ###
%
\begin{equation}
\label{Spjotvoll MDF}
\bolds{\delta}_{\mathrm{SPJ}}(\alpha) =
\mathop{\arg\max}_{\bolds{\delta} \in\mathcal{D}_0}
\{\operatorname{ETP}(\bolds{\delta},Q_1)\dvtx
\operatorname{EFP}(\bolds{\delta},Q_0) \le\alpha
\},
\end{equation}
where $Q_0 \in\mathcal{Q}_0$ and $Q_1 \in\mathcal{Q}_1$;
whereas the optimal discovery procedure (ODP)
in \cite{Sto07} is
%
%e5.5 ###
%
\begin{equation}
\label{Storey ODP}
\bolds{\delta}_{\mathrm{STO}}(\alpha;Q) =
\mathop{\arg\max}_{\bolds{\delta} \in\mathcal{D}}
\{\operatorname{ETP}(\bolds{\delta},Q)\dvtx
\operatorname{EFP}(\bolds{\delta},Q) \le\alpha
\},
\end{equation}
where $Q$ is the true probability measure of ${X}$.
The use of EFP as type I error measure in \cite{Sto07} enabled
a calculus of variations optimization to obtain the ODP.
This has a particularly interesting structure when we
utilize as its input the vector of $p$-value statistics
$(S_m^*(x_m,u_m)\dvtx m \in\mathcal{M})$ from the
MP MDP $\bolds{\Delta}^* = (\Delta_m^*\dvtx m \in\mathcal{M})$
with multiple decision size function $\mathbf{A}_{\bolds{\Delta}^*}^*
= \{(A_m^*(\eta)\dvtx \eta\in[0,1])\dvtx m \in\mathcal{M}\}$
and multiple decision ROC function
$\bolds{\rho}_{\bolds{\Delta}^*}^* = \{(\rho_m^*(\eta) \dvtx
\eta\in[0,1])\dvtx m \in\mathcal{M}\}$
and with $A_m^*(\cdot)$ and $\rho_m^*(\cdot)$ both differentiable
with derivatives $(A_m^*)^\prime(\cdot)$ and $(\rho_m^*)^\prime
(\cdot)$.
The significance thresholding function
$\mathcal{S}\dvtx ([0,1],\sigma[0,1]) \rightarrow(\Re,\sigma(\Re))$
utilized in the ODP becomes
%
%e5.6 ###
%
\begin{equation}
\label{significance thresholding function}
\mathcal{S}(s;Q) = \frac{\sum_{m \in\mathcal{M}_1(Q)}
(\rho_m^*)^\prime(s)}{\sum_{m \in\mathcal{M}_0(Q)} (A_m^*)^\prime(s)},
\end{equation}
a consequence of Lemma 2 in \cite{Sto07} and Proposition \ref
{prop-distributions of p value}.
The ODP $\bolds{\delta}_{\mathrm{STO}} = (\delta_{m,\mathrm{STO}}\dvtx m \in\mathcal
{M})$ has
a single-thresholding structure with components
\[
%%\label{components of ODP}
\delta_{m,\mathrm{STO}}(S_m^*(x_m,u_m);Q) = I\{\mathcal{S}(S_m^*(x_m,u_m);Q)
\ge\lambda\},\qquad
m \in\mathcal{M},
\]
where $\lambda\in[0,\infty)$ is chosen so the size constraint
on $\operatorname{EFP}(\bolds{\delta}_{\mathrm{STO}}(\alpha;Q),Q)$ is
{approximately} satisfied.
Observe that each of these components is {still} of simple-type,
unless $\lambda$ is determined in a data-dependent manner
using the full data $({x},{u})$.
Note also that $\bolds{\delta}_{\mathrm{STO}}$ was derived under complete
knowledge of the unknown $Q$,
or more specifically, the sets $\mathcal{M}_0(Q)$ and $\mathcal
{M}_1(Q)$, as can be seen in
(\ref{significance thresholding function}), hence is referred to as an
oracle MDF. For the simple null versus simple
alternative hypotheses case, the size functions $A_m^*(\cdot)$'s and
the ROC functions
$\rho_m^*(\cdot)$'s will be known, but with composite hypotheses they
will be unknown.
To implement $\bolds{\delta}_{\mathrm{STO}}$, it was proposed in
\cite{Sto07,StoDaiLee07} that these unknown
quantities, sets, functions, or significance thresholding function,
be estimated using the data $({x},{u})$. This will make the estimated
ODP of compound type.
But note that through this plug-in approach the exact optimality property
of the ODP need not anymore hold for the estimated version;
see also \cite{SunCai07,FerFriRueThoKon08}.
In contrast, $\bolds{\delta}_{\mathrm{SPJ}}$ is determined only by the
two classes of extreme probability measures, $\mathcal{Q}_0$ and
$\mathcal{Q}_1$,
so the marginal probability measures, $Q_m$'s, are completely known,
and not by the unknown
true probability measure $Q$ governing $X$. This fact was criticized in
\cite{Sto07} as a ``potentially problematic optimality'' criterion.
More importantly, it should be recognized that both $\bolds{\delta
}_{\mathrm{SPJ}}$ and
$\bolds{\delta}_{\mathrm{STO}}$ need not be
the optimal weak or strong FWER- or FDR-controlling MDFs since
the Bonferroni upper bound for $R_{0}(\bolds{\delta},Q)$ utilized
in their derivations is hardly a sharp upper bound.

The criticism leveled against $\bolds{\delta}_{\mathrm{SPJ}}$ could also
be invoked against our optimal weak FWER-controlling procedure since we also
relied on a criterion determined only by the extreme classes
$\mathcal{Q}_0$ and $\mathcal{Q}_1$. However, note that {each}
component of the
optimal weak FWER-controlling multiple decision size vector, and consequently
each component of $\bolds{\delta}_W^*(\alpha)$, uses {all} of the
$Q_{m0}$'s and $Q_{m1}$'s, analogously to the ODP, though the MDF
$\bolds{\delta}_W^*(\alpha)$ is still neither adaptive nor compound.
Our development of this simple MDF,
which is optimal in the class $\mathcal{D}_0$, is a prelude
to our development of adaptive and compound MDFs \textit{strongly}-controlling
FWER and~FDR. The MDF $\bolds{\delta}_W^*(\alpha)$ will be the anchor
for these FWER and FDR strongly-controlling compound MDFs. These new
MDFs are discussed in Section \ref{section-strong FWER control} for
strong FWER-control and in Section \ref{section-FDR procedures} for
FDR control. Our approach to obtaining these
strongly-controlling MDFs is indirect, whereas
that in \cite{Sto07} is direct.
There is also an intrinsic difference in the problems considered
since our focus is on the type I error risk functions $R_{0}$ and $R_{1}$,
whereas in \cite{Spj72,Sto07} the simpler
type I error metric of EFP was utilized. Looking forward,
though our starting point is the optimal weak FWER-controlling simple MDF
$\bolds{\delta}_W^*(\alpha)$, there
is confidence in the viability of our indirect approach to generate
good MDFs since
we will establish later that both the sequential {\v{S}}id{\'a}k
procedure and the
BH procedure are special cases of our new MDFs under exchangeability.

%s5.2 ###
\subsection{Families with MLR property}
\label{subsection-MLR Families}

The initial simplification to the
simple null versus simple alternative hypotheses for
each $m \in\mathcal{M}$ could be perceived as a limitation
because of the need to know the $Q_{m1}$'s to
determine the ROC functions.
However, this approach, which was also implemented
in \cite{Spj72,Sto07,RoqWie09}, is natural and
historically-justified by the Neyman--Pearson
framework. We surmise that
in this multiple decision problem, the solution to the simple null
versus simple
alternative hypotheses setting will play a prominent role in solving
the composite hypotheses setting, since it appears that for an MDF to
possess optimality, it will require knowledge, either in exact,
approximate, or
estimated forms, of the alternative hypotheses distributions.
We touch on this aspect in the presence of the
monotone likelihood ratio (MLR) property; see \cite{Leh97}.

Suppose that for each $m \in\mathcal{M}$, the density function $q_m$
belongs to a one-dimensional parametric family
$\mathcal{F}_m = \{q_m(\cdot;\xi_m)\dvtx\xi_m \in
\Gamma_m \subset\Re\}$ which possesses the MLR property.
A typical pair of hypotheses to be tested would
be $H_{m0}^*\dvtx \xi_m \le\xi_{m0}$ versus $H_{m1}^*\dvtx \xi_m > \xi_{m0}$,
where $\xi_{m0}$ is known.
With the MLR property, a uniformly most powerful (UMP) test function
$\delta_m(X_m,U_m;\eta_m)$ of size $\eta_m$ exists,
with this UMP test identical to the MP test of
size $\eta_m$ for the simple null hypothesis
$H_{m0}\dvtx \xi_m = \xi_{m0}$ versus
the simple alternative hypothesis $H_{m1}\dvtx \xi_m = \xi_{m1}$,
with $\xi_{m1} > \xi_{m0}$. When dealing with the single-pair
hypothesis testing problem, recall that exact knowledge of the value of
$\xi_{1}$ is not necessary since the critical constants of the
size-$\eta$ MP
test for $H_0\dvtx \xi= \xi_0$ versus $H_1\dvtx \xi= \xi_1$ can be made
independent of
$\xi_{1}$. In contrast, for the multiple decision problem, to determine
the optimal size allocations for each of the $M$ MP tests,
the powers of the tests at the $\xi_{m1}$'s are required, hence
the need to know the values of the $\xi_{m1}$'s.
When $M$ is large, such information may not be so forthcoming.
The default procedure is the simplistic approach of simply
assuming that the $(Q_{m0},Q_{m1})$ is invariant in $m$,
which is the {exchangeable setting}. However, this exchangeable
assumption is most likely wrong as a consequence of varied effect sizes
or different test functions utilized.
See, for instance, \cite{EfrAAS08} for real situations where
exchangeability do not hold.
We propose two possible solutions to this dilemma.

The first approach is to solicit from the scientific investigator the
values of
the $\xi_{m1}$'s for which the powers are of most interest.
Such values may coincide with those that are scientifically
different from the $\xi_{m0}$'s. Such elicitation, which may not be
very feasible in practice if $M$ is large, but which may be made
possible by
forming subclasses or clusters of the $M$ genes as in \cite{EfrAAS08},
amounts to specifying {effect sizes}. Formation of such
clusters must be made in close consultation with the investigator, or perhaps
guided by the result of a preliminary cluster analysis using data
independent of that used in the decision functions.
For the specified $\xi_{m1}$'s, the ROC functions in the determination
of the optimal
weak FWER-controlling multiple size vector become $\rho_m(\eta) =
\pi_{\delta_m^*(\eta)}(\xi_{m1})$ for $m \in\mathcal{M}$,
where $\delta_{m}^*(\eta)$ is the simple MP test of size $\eta$ for testing
$H_{m0}\dvtx \xi_m = \xi_{m0}$ versus $H_{m1}\dvtx \xi_m = \xi_{m1}$, and
$\pi_{\delta_{m}^*(\eta)}(\xi_{m1})$ is the power of $\delta
_{m}^*(\eta)$
(at $\xi_m = \xi_{m1}$). In the clustered situation with $\mathcal
{M} = \biguplus_{k=1}^K
\mathcal{M}_k$, we may denote by $\bar{\rho}_k(\eta)$ and $\zeta
_k$, respectively,
the common ROC function and size for the decision functions in cluster
$\mathcal{M}_k$.
Under second-order differentiability of $\bar{\rho}_k(\eta)$'s, by
Theorem \ref{theo-optimal sizes}, the optimal weak FWER-$\alpha$ controlling
multiple size vector $\zeta(\alpha) = (\zeta_1(\alpha),\zeta
_2(\alpha),\ldots,
\zeta_K(\alpha))$ is the $\zeta= (\zeta_1,\zeta_2,\ldots,\zeta
_K)$ that
solves the set of equations
$\forall k=1,2,\ldots,K\dvtx\bar{\rho}_k^\prime(\zeta_k)(1 -
\zeta_k) = \lambda$
for some $\lambda\in\Re_+$ with
$\sum_{k=1}^K |\mathcal{M}_k| \log(1 - \zeta_k) = \log(1-\alpha)$.

The second approach, analogous to those in
\cite{WesKriYou98,RubDudLaa06,SunCai07,Sto07,StoDaiLee07,KanYeLiuAllGao09}
is to estimate or approximate the underlying values of the $\xi_m$'s
either using the observed data $x$, possibly via shrinkage-type estimators,
or through the use of prior information which could be informed by external
covariates as in \cite{FerFriRueThoKon08}.
Addressing this same restriction of requiring knowledge of the simple
null and simple
alternative hypotheses and advocating this second approach,
\cite{RoqWie09}, page 679, stated:
%%
%%\begin{quote}
``although leading to oracle procedures, it can be used in practice as
soon as the null
and alternative distributions are estimated or guessed reasonably
accurately from
independent data.''
%%\end{quote}
%%
By ``independent data'' is meant in \cite{RoqWie09} as data different
from that
used in performing the actual tests. However, such {external} data
need not always be used for estimating
or imputing the unknown parameters. For example, suppose that for
each $m \in\mathcal{M}$, data $x_m$ could be partitioned into $(v_m,w_m)$.
We may then use $\tilde{\xi}_m(v_m) = \max\{\xi_{m0},\hat{\xi
}_m(v_m)\}$,
where $\hat{\xi}_m(v_m)$
is the maximum likelihood estimate of $\xi_m$ based on $v_m$, and proceed
as in the preceding paragraph with $\xi_{m1}$ set to $\tilde{\xi
}_{m}(v_m)$ for each
$m \in\mathcal{M}$, and with the component data $w_m$ used in the
test functions.
The resulting MDF will be of an adaptive type,
possibly also compound as in \cite{SunCai07}
if shrinkage estimators are used for estimating the $\xi_m$'s using the
$v_m$ components.
Observe that if for some $m_0 \in\mathcal{M}$,
$\tilde{\xi}_{m}(v_{m_0})$ and $\xi_{{m_0}0}$ are very close or identical,
then a relatively small size will be allocated to the MP test for
component $m_0$.
This amounts to downgrading the testing problem for this component,
a fact of importance since a criticism of multiple hypotheses testing,
especially
when using FDR, is that an unscrupulous investigator may keep adding
irrelevant genes. When using the adaptive MDF arising from the optimal
multiple decision size vector, this investigator's strategy will backfire
since the adaptive MDF will automatically downgrade the irrelevant genes.
This second approach still requires deeper study.
For instance, there is the issue of how to partition
each $x_m$ into the $v_m$ and $w_m$ components. Furthermore, the impact
of a
misspecified $\xi_{m1}$, possibly arising from the estimation
procedure, needs
to be ascertained.

%s5.3 ###
\subsection{Connections to $p$-value statistics}
\label{subsection-In Terms of P-Values}

Proposition \ref{prop-distributions of p value} indicates that the ROC function
$\eta\mapsto\rho_m(\eta)$ is differentiable if and only if
the distribution function of the $p$-value statistic
$S_m(X_m,U_m)$ under $H_{m1}\dvtx Q_m = Q_{m1}$
is differentiable. In this case, $\rho_m^\prime(\cdot)$ coincides with
$h_m(\cdot)$, the density function of
$S_m(X_m,U_m)$ under $H_{m1}\dvtx Q_m = Q_{m1}$.
Condition (i) in Theorem \ref{theo-optimal sizes} is
equivalent to the constancy in $m$ of $h_m(\eta_m)(1-\eta_m)$.
%
%%\begin{equation}
%%\label{condition of optimality in p-value density}
%%h_m(\eta_m) (1 - \eta_m) = \mbox{Constant},\ \forall m \in
%%\end{equation}
%
This is surprising since it indicates that it is \textit{not} enough to simply
find the sizes that maximize these $h_m(\cdot)$'s,
as dictated by the Neyman--Pearson lemma when dealing
with a single pair of null and alternative hypotheses.
Rather, in the multiple hypotheses testing scenario,
there is attenuation in that larger sizes
incur penalties. Condition (i) in Theorem \ref{theo-optimal sizes}
governs the interactions among the $M$ tests regarding their size allocations
to achieve the best overall result, in terms of overall type II error,
among themselves.

The optimal weak FWER-controlling MDF
can be converted to a procedure based on the $p$-value statistics.
If $\bolds{\eta}^*(\alpha) = (\eta_m^*(\alpha), m \in\mathcal{M})$
is the optimal weak FWER-$\alpha$ multiple decision
size vector and $(S_m(x_m,u_m), m \in\mathcal{M})$
is the vector of computed $p$-value statistics, the decision based on
data $(x,u) = ((x_m,u_m), m \in\mathcal{M})$ is
$\delta^*(x,u) = (I\{S_m(x_m,u_m) \le\eta_m^*(\alpha)\}, m \in
\mathcal{M})$,
an MDF based on weighted $p$-values. This is
related to the approach in several papers using weighted $p$-values
such as \cite{GenRoeWas06,WasRoe06,RubDudLaa06,KanYeLiuAllGao09,RoqWie09}.
In our case, the weights are tied-in to the optimal sizes.

%s6 ###
\section{\textit{Strong} FWER control}
\label{section-strong FWER control}

Let $\bolds{\Delta}^* =
(\Delta_m^*, m \in\mathcal{M})$ be the MP MDP with
$\Delta_m^* = (\delta_m^*(\eta)\dvtx \eta\in[0,1])$ the MP decision process
for $H_{m0}\dvtx Q_m = Q_{m0}$ versus $H_{m1}\dvtx Q_m = Q_{m1}$
based on $(X_m,U_m)$.
Wlog, assume that the size function $A_m(\cdot)$ of
$\Delta_m^*$ satisfies $A_m(\eta) = \eta$. Define
$\bolds{\eta}\dvtx [0,1] \rightarrow[0,1]^M$ such that
$\bolds{\eta}(\alpha) = (\eta_m(\alpha), m \in\mathcal{M})$ is
the optimal weak
FWER-controlling multiple decision size vector at level $\alpha$.
Assume that each component of this mapping is nondecreasing
and continuous, which is the case when the ROC functions
of $\bolds{\Delta}^*$ are twice-differentiable
as established in Proposition \ref{prop-monotonicity of optimal solution}.

For a weak FWER threshold
of $\alpha\in[0,1]$, the optimal MDF in $\mathcal{D}_0$ is
%
%%\begin{equation}
%%\label{FWER weak MDF}
$\bolds{\delta}_W^*(\alpha) = (\delta_m^*(\eta_m(\alpha)), m \in
\mathcal{M})$,
%%\end{equation}
%
as given in (\ref{optimal weak FWER controlling MDF}).
Associated with this MDF is the \textit{generalized} multiple decision
$p$-value statistic
$\mathbf{W} = (W_m, m \in\mathcal{M})$, where
%
%e6.1 ###
%
\begin{equation}
\label{generalized p-value}
W_m \equiv W_m(X_m,U_m)
= \inf\{\alpha\in[0,1]\dvtx \delta_m^*(\eta_m(\alpha)) = 1\}.
\end{equation}
The $w_m = W_m(x_m,u_m)$ is the smallest weak FWER size leading to
rejection of $H_{m0}$ when using $\bolds{\delta}_W^*(\alpha)$ given data
$(x,u) = ((x_m,u_m), m \in\mathcal{M})$. The
usual $p$-value statistic $S_m$ [see (\ref{defn-p value statistic})]
for $\delta_m^*$ is related to $W_m$ via
%
%e6.2 ###
%
\begin{equation}
\label{p value relationships}
\forall m \in\mathcal{M}\dvtx S_m(X_m,U_m) = \eta_m(W_m(X_m,U_m)).
\end{equation}
%
%%In the sequel, the statistic $\mathbf{Q}$, which takes values
%%in the set of permutations of $(1,2,\ldots,M)$, and
%%denoted by
%
%%\begin{equation}
%%\label{antiranks of generalized p-values}
%%\mathbf{Q} = (Q_1,Q_2,\ldots,Q_M) = ((1),(2),\ldots,(M)),
%%\end{equation}
%
%%represents the anti-rank vector of $\mathbf{W}$, so that
%%$W_{Q_1} < W_{Q_2} < \ldots< W_{Q_M}$ or, equivalently,
%%$W_{(1)} < W_{(2)} < \ldots< W_{(M)}$.

Now, a l\'a \cite{Sto07,SunCai07},
suppose an Oracle knows $Q$, the true
underlying probability measure of $X$. For the MDF $\bolds{\delta
}_W^*(\alpha)$,
its FWER is
\[
R_{0}(\bolds{\delta}_W^*(\alpha),Q) =
1 - \prod_{m \in\mathcal{M}} [1 - \eta_m(\alpha)]^{1-\theta_m(Q)}.
\]
This is nondecreasing and continuous in $\alpha$ since
the mappings $\alpha\mapsto\eta_m(\alpha)$
for each $m \in\mathcal{M}$ are nondecreasing and continuous.
If the Oracle desires to control this type I error rate at a value $q^*
\in[0,1]$ and
also minimize the MDR given by
%
%%\begin{displaymath}
$R_2(\bolds{\delta}_W^*(\alpha),Q) = |\mathcal{M}_1(Q)| -
\sum_{m \in\mathcal{M}_1(Q)} \rho_m(\eta_m(\alpha))$,
%%\end{displaymath}
%
where $\rho_m(\eta_m(\alpha))$ is the power of $\delta_m^*(\eta
_m(\alpha))$,
then she should choose the largest $\alpha\in[0,1]$
such that $R_{0}(\bolds{\delta}_W^*(\alpha),Q) = q^*$.
Owing to the continuity and nondecreasing properties
of $R_{0}(\bolds{\delta}_W^*(\alpha),Q)$ in $\alpha$,
the Oracle's optimal $\alpha$ could also be expressed via
\[
\alpha^\dagger(q^*;Q) =
\inf\biggl\{\alpha\in[0,1]\dvtx \prod_{m \in\mathcal{M}}
[1 - \eta_m(\alpha)]^{1 - \theta_m(Q)} < 1 - q^*\biggr\}.
\]
%
%%with the convention that $\inf\varnothing= 1$.

However, there is no Oracle and
$Q$ is not known, else there is no multiple
decision problem. Thus, $\alpha^\dagger(q^*;Q)$ is not
observable. A natural idea is to estimate
the unknown $\theta_m(Q)$, the state of the $m$th pair of hypotheses.
An intuitive and simple estimator of $\theta_m(Q)$ for a fixed value of
$\alpha$ is
%
%e6.3 ###
%
\begin{equation}
\label{estimator of thetam(Q)}
\widehat{\theta}_m(Q) = \delta_m^*(\eta_m(\alpha)-) \equiv
\delta_m^*(X_m,U_m;\eta_m(\alpha)-).
\end{equation}
In turn, we obtain a \textit{step-down} estimator
$\alpha^\dagger(q^*) \equiv\alpha^\dagger(X,U;q^*)$
of the Oracle-based $\alpha^\dagger(q^*;Q)$ given by
%
%e6.4 ###
%
\begin{equation}
\label{alpha dagger}
\alpha^\dagger(q^*) =
\inf\biggl\{\alpha\in[0,1]\dvtx \prod_{m \in\mathcal{M}}
[1 - \eta_m(\alpha)]^{1 - \delta_m^*(\eta_m(\alpha)-)} < 1 -
q^*\biggr\}.
\end{equation}
This determines a compound
MDF $\bolds{\delta}_S^*(q^*) \equiv\bolds{\delta}_S^*(X,U;q^*)
\in\mathcal{D}$,
where
%
%e6.5 ###
%
\begin{equation}
\label{FWER strong MDF}
\bolds{\delta}_S^*(q^*) =
\bigl(\delta_m^*(\eta_m(\alpha^\dagger(q^*))), m \in\mathcal{M}\bigr).
\end{equation}
By virtue of the optimal choice of the $\eta_m(\alpha)$'s
and the use of the MP tests, we expect $\bolds{\delta}_S^*(q^*)$
to possess excellent, if not optimal, MDR-properties.
By taking the infimum over the weak FWER-size $\alpha$
coupled with the estimation of $\theta_m(Q)$
by $\delta_m^*(\eta_m(\alpha)-)$ in (\ref{alpha dagger}),
there occurs an adaptive downweighting of
components whose $H_{m0}$'s are most likely correct as dictated by
the data $(x,u)$.
Theorem \ref{theo-strong FWER} below establishes that
$\bolds{\delta}_S^*(q^*)$ in (\ref{FWER strong MDF})
does strongly control the FWER.
\begin{theo}
\label{theo-strong FWER}
Let $q^* \in[0,1]$. Then, $\forall Q \in\mathcal{Q}$,
$R_{0}(\delta_S^*(q^*),Q) \le q^*$.
\end{theo}

%%\begin{proof}
%%See Proof of Theorem \ref{theo-strong FWER} in \cite{PenHabWu10Supp}.
%%\end{proof}

Next, we reexpress $\bolds{\delta}_S^*(q^*)$ in terms of the generalized
$p$-value statistic $\mathbf{W}$. This is achieved by
defining the random variable
\[
J^\dagger(q^*) %%& = &
%%\max\{j \in\mathcal{M}\dvtx T_1(W_{(i)}) \ge1 - q^*,\ i=1,2,
%%\} \\
= \max\Biggl\{j \in\mathcal{M}\dvtx\prod_{m=i}^M
\bigl[1 - \eta_{(m)}\bigl(W_{(i)}\bigr)\bigr] \ge1 - q^*, i=1,2,\ldots,j\Biggr\}.
\]
Since $\alpha^\dagger(q^*) \in[W_{(J^\dagger(q^*))},
W_{(J^\dagger(q^*)+1)})$, then
\[
\bolds{\delta}_S^*(q^*) = \bigl(
\delta_m^*\bigl(\eta_m\bigl(W_{(J^\dagger(q^*))}\bigr)\bigr), m \in\mathcal{M}
\bigr).
\]
The next result shows that the sequential step-down {\v{S}}id{\'a}k MDF,
which strongly controls FWER,
is a special case of $\bolds{\delta}_S^*(q^*)$ under exchangeability.
\begin{prop}
\label{prop-sequential Sidak}
If the $M$ ROC functions are identical, then $\bolds{\delta
}_S^*(q^*)$ coincides with
the sequential {\v{S}}id{\'a}k step-down FWER-controlling MDF.
\end{prop}

%%\begin{proof}
%%See Proof of Corollary \ref{coro-sequential Sidak} in
%%\end{proof}

%s7 ###
\section{\textit{Strong} FDR control}
\label{section-FDR procedures}

Assume the same framework as in Section \ref{section-strong FWER control}.
Our idea in obtaining an FDR-controlling MDF builds on the development
of the BH MDF, specifically the rationale of Theorem 2 in \cite{BenHoc95}.
Let $q^* \in[0,1]$ be the desired
FDR threshold and $Q$ be the underlying probability measure of $X$.
We introduce two stochastic processes: $\mathbf{T}_0 =
\{T_0(\alpha;Q)\dvtx \alpha\in[0,1]\}$ and
$\mathbf{T} = \{T(\alpha)\dvtx \alpha\in[0,1]\}$,
where
\[
T_0(\alpha;Q) = \sum_{m \in\mathcal{M}_0(Q)} \delta_m^*(\eta
_m(\alpha)) \quad\mbox{and}\quad
T(\alpha) = \sum_{m \in\mathcal{M}} \delta_m^*(\eta_m(\alpha)).
\]
For the MDF $\delta_W^*(\alpha)$, its FDR is
\[
R_1(\delta_W^*(\alpha),Q) =
E_Q\biggl\{\frac{T_0(\alpha;Q)}{T(\alpha)} I\{T(\alpha) > 0\}
\biggr\}.
\]
By the definition of the generalized $p$-value statistics $W_m$'s in
(\ref{generalized p-value}),
we have for $\alpha\in[W_{(m)},W_{(m+1)})$ that $T(\alpha) = m$, whereas
%
%e7.1 ###
%
\begin{equation}
\label{bounding}
E_Q\{T_0(\alpha;Q)\} = \sum_{m \in\mathcal{M}} \bigl(1 - \theta_m(Q)\bigr)
\eta_m(\alpha)
\le\sum_{m \in\mathcal{M}} \eta_m(\alpha).
\end{equation}
Focus now on an $\alpha\in[W_{(m)},W_{(m+1)})$. If
$\sum_{j \in\mathcal{M}} \eta_j(W_{(m)}) \le m q^*$, then the best
$\alpha$ in this interval will be the largest value satisfying
$\sum_{j \in\mathcal{M}} \eta_j(\alpha) \le m q^*$, since by increasing
$\alpha$, the MDR decreases as argued in the development of $\delta_S^*(q^*)$
in Section \ref{section-strong FWER control}. This motivates our
definition of $\alpha^*(q^*) = \alpha^*(X,U;q^*)$ as the \textit{step-up}
estimator
%
%e7.2 ###
%
\begin{equation}
\label{alpha-star}
\alpha^*(q^*) =
\sup\biggl\{\alpha\in[0,1]\dvtx \sum_{m \in\mathcal{M}} \eta
_m(\alpha)
\le q^* \sum_{m \in\mathcal{M}} \delta_m^*(\eta_m(\alpha))
\biggr\}.
\end{equation}
This induces a compound MDF $\bolds{\delta}_F^*(q^*) \equiv\bolds
{\delta}_F^*(X,U;q^*) \in
\mathcal{D}$ given by
%
%e7.3 ###
%
\begin{equation}
\label{MDF FDR control}
\bolds{\delta}_F^*(q^*) = \bigl(\delta_m^*(\eta_m(\alpha^*(q^*))), m
\in\mathcal{M}\bigr).
\end{equation}
Theorem \ref{theorem-strong FDR control of procedure} establishes that
$\bolds{\delta}_F^*(q^*)$ does control the FDR at $q^*$.
Interestingly, the proof of this theorem, which can be found
in \cite{PenHabWu10Supp}, employs a reverse martingale argument.
\begin{theo}
\label{theorem-strong FDR control of procedure}
Let $q^* \in[0,1]$. If, $\forall Q \in\mathcal{Q}\setminus\{Q_0\}$
and $\forall\alpha\in(0,1)$,\break
%
%%\begin{equation}
%%\label{size condition}
$|\mathcal{M}_0(Q)| \max_{m \in\mathcal{M}_0(Q)} \eta_m(\alpha)
\le
\sum_{m \in\mathcal{M}} \eta_m(\alpha)$,
%%\end{equation}
%
then $R_1(\bolds{\delta}_F^*(q^*),Q) \le q^*$ for $\forall Q \in
\mathcal{Q}$.
\end{theo}

%%\begin{proof}
%%See Proof of Theorem \ref{theorem-strong FDR control of procedure} in
%%\end{proof}

%%\begin{coro}
%%If the conditions of Theorem \ref{theo-optimal sizes} and condition (
%%in Theorem \ref{theorem-strong FDR control of procedure} hold, then
%the conclusion of Theorem
%%\ref{theorem-strong FDR control of procedure} holds.
%%\end{coro}
%%
%%\begin{proof}
%%Follows from Theorem
%%\ref{theorem-strong FDR control of procedure} and
%%Proposition \ref{prop-monotonicity of optimal solution} since the
%latter guarantees
%%that the mappings $\alpha\mapsto\eta_m(\alpha)$ for $m \in
%%are nondecreasing.
%%\end{proof}

Some remarks are in order regarding the condition in Theorem
\ref{theorem-strong FDR control of procedure}.
Clearly, the {\v{S}}id{\'a}k multiple decision size vector, which is the
optimal multiple decision size vector
when the ROC functions are identical, always satisfies this condition.
When not in this exchangeable setting, this condition
induces some control on the differences of the ROC functions.
%%We conjecture that a weaker condition is possible to still achieve
%%FDR-control by the MDF $\bolds{\delta}_F^*$, though a
%%non-martingale-based proof may be needed to establish such a result.
The next proposition establishes that the BH procedure is
a special case of $\bolds{\delta}_F^*(q^*)$ under exchangeability.
\begin{prop}
\label{prop-BH procedure as special case}
If the ROC functions are identical, then $\bolds{\delta}_F^*(q^*)$
is the
FDR-$q^*$ controlling MDF in \cite{BenHoc95}.
\end{prop}

%%\begin{proof}
%%See Proof of Corollary \ref{coro-BH procedure as special case} in
%%\end{proof}

Examination of the proof of Proposition \ref{prop-BH procedure as
special case}
as presented in \cite{PenHabWu10Supp} shows
that the BH MDF $\bolds{\delta}^{\mathrm{BH}}(q^*)$
coincides with the {\v{S}}id{\'a}k-size based MDF $\bolds{\delta}^S(q^*)$.
The martingale proof for
Theorem \ref{theorem-strong FDR control of procedure}
thus carries over to establishing
FDR control by $\bolds{\delta}^{\mathrm{BH}}(q^*)$.
We mention that a martingale-based proof of FDR control
by $\bolds{\delta}^{\mathrm{BH}}(q^*)$ has also been presented in \cite{StoTaySie04}.
%%We would like to emphasize also that
%%the compound MDF $\bolds{\delta}_F^*(q^*)$ do achieve FDR control,
%%in contrast to the \textit{estimated} ODP MDF in \cite{Sto07} which
%%only {approximately} satisfies FDR control.
%

We also provide an alternative form of $\bolds{\delta}_F^*(q^*)$ in
terms of the
generalized $p$-value statistics $W_m$'s, a form analogous to
the conventional formulation of the BH procedure. Define
%
%e7.4 ###
%
\begin{equation}
\label{FDR extension cut}
J^*(q^*) \equiv J^*(X,U;q^*) = \max\Biggl\{m \in\mathcal{M}\dvtx
\sum_{j \in\mathcal{M}} \eta_j\bigl(W_{(m)}\bigr) \le q^* m
\Biggr\}.
\end{equation}
Then, it is easy to see that $\bolds{\delta}_F^*(q^*)$ rejects
$H_{(m)0}$ for
$m \in\{1,2,\ldots,J^*(q^*)\}$ and accepts $H_{(m)0}$ for
$m \in\{J^*(q^*)+1,J^*(q^*)+2,\ldots,M\}$.

Finally, let us examine further the generalized $p$-value statistics $W_m$'s.
Focusing on $W_{(1)}$, under $Q_0$, we have that, for $a \in(0,1)$,
\[
\Prr_{Q_0}\bigl(W_{(1)} > a\bigr) = \Prr_{Q_0}\biggl\{\bigcap_{m \in
\mathcal{M}}[
\delta_m^*(\eta_m(a)) = 0]\biggr\} %%\\ & = &
= \prod_{m \in\mathcal{M}} [1 - \eta_m(a)] = 1 - a,
\]
the second equality obtained by using the independence of the $\delta
_m^*$'s under $Q_0$.
Thus, $W_{(1)}$ is standard uniform when all null hypotheses
are correct. Using this uniformity result
and Lemma D.2 presented in \cite{PenHabWu10Supp} dealing with lower
and upper bounds
of $\eta_\bullet$ for $\eta\in UB(C_\alpha)$, we obtain in
Proposition \ref{prop-lower and upper bounds} presented below
a lower bound for $R_1(\delta_F^*(q^*),Q_0)$,
the FDR when all the null hypotheses are correct.
\begin{prop}
\label{prop-lower and upper bounds}
$\forall q^* \in[0,1]$, $1 - (1-q^*/M)^M \le R_1(\delta_F^*(q^*),Q_0)
\le q^*$.
\end{prop}

%%\begin{proof}
%%See Proof of Proposition \ref{prop-lower and upper bounds} in
%%\end{proof}

%s8 ###
\section{A modest simulation}
\label{section-Simulated Comparison}

We compared through computer simulations the performances of
$\bolds{\delta}_F^*$ and $\bolds{\delta}^{\mathrm{BH}}$ in terms of FDR
and MDR.
The simulation model utilized is similar to the Gaussian example
illustrating the optimal weak FWER-controlling
procedure in Section \ref{subsection-Concrete Examples}.
In this model, the observables are
$X_m \sim N(\mu_m,1)$ for each $m \in\mathcal{M}$,
which are independent of each other. The $m$th pair of hypotheses is
$H_{m0}\dvtx \mu_m \le0$ versus $H_{m1}\dvtx \mu_m > 0$. The UMP
size-$\eta
_m$ test is
$\delta_m^*(X_m;\eta_m) = I\{X_m > \Phi^{-1}(1-\eta_m)\}$. The true
values of the means
$\mu_m$'s are $\mu_m = \xi_m \theta_m, m \in\mathcal{M}$,
with $\theta_m \sim \operatorname{Ber}(p)$ and effect sizes $\xi_m \sim|N(\nu,1)|$,
again independently generated from each other. The parameter
combinations were induced by taking
$M \in\{20,50, 100\}$, $p \in\{0.1, 0.2, 0.4\}$
and $\nu\in\{1, 2, 4\}$. The FDR-threshold utilized were $q^* \in\{
0.05, 0.10\}$.
Since the computational implementation of $\bolds{\delta}_F^*$ takes time,
for each combination of $(q^*, M, \nu, p)$, we limited our simulations to
1,000 replications.
The simulated FDR and MDR$^*$ were the averages
of the false discovery proportions, $L_1(a,Q)$'s, and
the standardized missed discovery proportions,
$L_2(a,Q)/|\mathcal{M}_1(Q)|$, over the 1,000 replications.
We used this standardized MDR since, for each replicate,
a $Q$ is generated, hence $|\mathcal{M}_1(Q)|$ differs over the
replications. In essence, we are comparing the averages of
$R_2(\delta_F^*,Q)/|\mathcal{M}_1(Q)|$ and $R_2(\delta
^{\mathrm{BH}},Q)/|\mathcal{M}_1(Q)|$,
where the averaging is with respect to the mechanism generating the
$Q$'s over the
simulation replications.

We only report results for $q^* = 0.10$ in Table \ref{table1000.10}
since results for $q^* = 0.05$ lead to similar conclusions.
From this table, we
observe that both $\bolds{\delta}_F^*$ and
$\bolds{\delta}^{\mathrm{BH}}$ fulfill the FDR-constraint, and
in a conservative manner, which is expected from theory.\vspace*{1pt}
More importantly, the MDR-performance
of $\bolds{\delta}_F^*$ is better compared to that
of $\bolds{\delta}^{\mathrm{BH}}$, with this dominance holding
for all twenty-seven parameter combinations. Observe that as $M$ is increased
with $(\nu,p)$ remaining the same, there is an increase in their MDR$^*$'s;
whereas, when $\nu$ is increased, which increases the effect sizes,
their MDR$^*$'s decrease.
Interestingly, the impact of a change of value in $p$, the proportion
of true alternative hypotheses, did not necessarily translate into
a monotone change in their MDR$^*$'s, especially when $M = 20$, though
for the larger
$M$-values, the change in MDR$^*$ appears monotonically decreasing.

%%\input{Table200.05}
%%\input{Table200.10}

%t1 ###
%
\begin{table}
\caption{Comparison of the false discovery rate (FDR)
and standardized missed discovery rate (MDR$^*$) performance of
MDFs $\delta_F^*$ and $\delta^{\mathrm{BH}}$ under a variety of simulation
parameters. This table is for $q^* = 0.10$. The FDR and MDR$^*$ are in
percentages. The number of replications is 1,000} \label{table1000.10}
\begin{tabular*}{\tablewidth}{@{\extracolsep{\fill}}d{2.0}cd{3.0}cccd{2.2}d{1.2}d{2.2}@{}}
\hline
& \multicolumn{1}{c}{$\bolds{q^*}$} & \multicolumn{1}{c}{$\bolds{M}$} & \multicolumn{1}{c}{$\bolds{\nu}$}
& \multicolumn{1}{c}{$\bolds{p}$} & \multicolumn{1}{c}{$\bolds{\delta_F^*}$\textbf{-FDR}}
& \multicolumn{1}{c}{$\bolds{\delta_F^*}$\textbf{-MDR}$\bolds{^*}$}
& \multicolumn{1}{c}{$\bolds{\delta^{\mathrm{BH}}}$\textbf{-FDR}}
& \multicolumn{1}{c@{}}{$\bolds{\delta^{\mathrm{BH}}}$\textbf{-MDR}$\bolds{^*}$} \\
\hline
1 & 0.1 & 20 & 1 & 0.1 & 8.03 & 70.80 & 8.43 & 72.64 \\
2 & 0.1 & 20 & 1 & 0.2 & 7.55 & 79.64 & 8.77 & 81.99 \\
3 & 0.1 & 20 & 1 & 0.4 & 6.05 & 77.47 & 6.65 & 80.30 \\
[3pt]
4 & 0.1 & 20 & 2 & 0.1 & 7.70 & 54.42 & 8.43 & 55.80 \\
5 & 0.1 & 20 & 2 & 0.2 & 7.39 & 56.32 & 7.59 & 57.31\\
6 & 0.1 & 20 & 2 & 0.4 & 6.47 & 47.82 & 6.21 & 49.38 \\
[3pt]
7 & 0.1 & 20 & 4 & 0.1 & 9.14 & 8.62 & 9.48 & 10.30 \\
8 & 0.1 & 20 & 4 & 0.2 & 7.80 & 7.34 & 6.97 & 9.20 \\
9 & 0.1 & 20 & 4 & 0.4 & 6.15 & 3.58 & 5.65 & 5.53 \\
[3pt]
10 & 0.1 & 50 & 1 & 0.1 & 8.83 & 84.87 & 9.26 & 87.05 \\
11 & 0.1 & 50 & 1 & 0.2 & 7.11 & 83.49 & 7.14 & 86.65 \\
12 & 0.1 & 50 & 1 & 0.4 & 6.45 & 78.91 & 6.42 & 82.30 \\
[3pt]
13 & 0.1 & 50 & 2 & 0.1 & 8.36 & 63.36 & 8.99 & 65.04 \\
14 & 0.1 & 50 & 2 & 0.2 & 8.74 & 57.30 & 8.73 & 58.93 \\
15 & 0.1 & 50 & 2 & 0.4 & 5.80 & 48.71 & 5.93 & 50.21 \\
[3pt]
16 & 0.1 & 50 & 4 & 0.1 & 8.84 & 10.28 & 8.93 & 12.09 \\
17 & 0.1 & 50 & 4 & 0.2 & 7.93 & 6.91 & 7.81 & 8.79 \\
18 & 0.1 & 50 & 4 & 0.4 & 6.34 & 3.40 & 6.07 & 5.68 \\
[3pt]
19 & 0.1 & 100 & 1 & 0.1 & 9.14 & 87.10 & 9.02 & 90.02 \\
20 & 0.1 & 100 & 1 & 0.2 & 8.21 & 84.05 & 8.78 & 87.38 \\
21 & 0.1 & 100 & 1 & 0.4 & 5.92 & 80.12 & 5.88 & 83.73 \\
[3pt]
22 & 0.1 & 100 & 2 & 0.1 & 9.79 & 66.10 & 9.24 & 67.93 \\
23 & 0.1 & 100 & 2 & 0.2 & 7.68 & 58.25 & 7.94 & 59.93 \\
24 & 0.1 & 100 & 2 & 0.4 & 5.74 & 49.29 & 6.10 & 50.90 \\
[3pt]
25 & 0.1 & 100 & 4 & 0.1 & 8.37 & 10.44 & 8.62 & 12.36 \\
26 & 0.1 & 100 & 4 & 0.2 & 7.72 & 5.93 & 7.81 & 8.22 \\
27 & 0.1 & 100 & 4 & 0.4 & 5.69 & 3.80 & 6.14 & 5.72 \\
\hline
\end{tabular*}
\end{table}

It may appear from this simulation study\vspace*{1pt}
that the standardized improvement of $\delta_F^*$ over $\delta^{\mathrm{BH}}$
is minuscule.
However, note that when translated
to overall number of discoveries, when $M$ is large, $\delta_F^*$ will
lead to many more
discoveries than $\delta^{\mathrm{BH}}$ while still maintaining desired FDR
control. Such an increase
in the number of discoveries may have important practical implications,
such as enlarging
the number of genes to be explored in consequent studies. This may translate
to enhanced chances of discovering crucial and important genes
without sacrificing the type I error rate.

%s9 ###
\section{Summary and concluding remarks}
\label{section-Concluding Remarks}

This paper provides some resolution on the role of the
individual powers of test or decision functions, more appropriately their
ROC functions, in multiple hypotheses testing problems.
The importance and relevance of these
problems have arisen because of the
proliferation of high-dimensional ``large $M$, small $n$'' data sets
in the natural, medical, physical, economic and social sciences.
Such data sets are being created or generated due to advances in
high-throughput technology, the latter fueled
by speedy developments in computer technology and miniaturization.

Almost a century ago, Neyman
and Pearson demonstrated the need to take into account the power
function and the
alternative hypothesis configuration when seeking an optimal test
procedure in
single-pair hypothesis testing. Their work led to a divorce from the
then-existing
significance or $p$-value approach. Currently, many multiple hypotheses
testing procedures, epitomized by the {\v{S}}id{\'a}k procedures for
weak and strong FWER control
and by the Benjamini--Hochberg (BH) procedure for FDR control,
are based on the $p$-values of the individual tests
and do not consider differences in
the power traits of the individual tests.
They are appropriate in so-called exchangeable
settings wherein power characteristics of the individual tests are identical.
Such settings, however, are more the exception than the rule, since
nonidentical power characteristics easily arise
due to differences in the effect sizes, the dispersion parameters, or
the test functions that are employed.

This paper examined whether differences in power characteristics
of the individual tests could be exploited
to improve on existing procedures for FWER and FDR control.
Procedures were developed under the historically most fundamental scenario
where the null and the alternative hypotheses are simple.
First, an optimal MDF within the class of simple MDFs
was shown to exist for weak FWER control.
This MDF is better than
the {\v{S}}id{\'a}k weak FWER-controlling MDF, though the latter is
a special case of the optimal MDF under exchangeability.
Optimality also informs us of an optimal size-investing strategy.
Second, by using this optimal, though still restricted, MDF as
an anchor, a compound MDF strongly
controlling FWER was obtained. The sequential {\v{S}}id{\'a}k MDF is a
special case of this
MDF under exchangeability. Third, we developed
a compound MDF that controls FDR. The
BH procedure obtains from this MDF under exchangeability.
By construction, these new MDFs have smaller MDRs relative
to those that did not exploit power differences.
The improvement was demonstrated through a modest simulation study
by comparing the new FDR-controlling MDF and the BH MDF.

Though the proposed MDFs do improve on existing
ones, we could not claim that
they are optimal among \textit{all} compound MDFs for
strong FWER or FDR control. This question of global
optimality is a difficult and elusive one.
So far none of the existing compound MDFs,
such as the {estimated} ODP in \cite{Sto07},
could claim global optimality.
In our case, the possible drawback is
that in constructing the new MDFs,
we started with the class of simple MDFs. The resulting MDFs
are indeed compound, but establishing global optimality
is not transparent. A question even arise as to whether there
truly exists an optimal MDF among all
compound MDFs that, say, control FDR. One thing certain about
our MDFs is that they do control FWER or FDR.
This is in contrast to some MDFs that are obtained from
oracle MDFs via plugging-in of estimates for unknown quantities.
Even though the oracle MDF, which are unimplementable,
satisfies the type I error rate control,
the plug-in step will usually invalidate such control.
See \cite{SunCai07} where optimality was in an asymptotic
sense and with the type I error rate being the mFDR, as well as
\cite{FerFriRueThoKon08,RoqWie09} for more discussions on these issues.

A natural layer to add in the decision-theoretic formulation of
the problem is a Bayesian layer where a prior measure is specified
on the unknown probability measure $Q$ or, alternatively, on $\theta(Q)$.
There is a possibility that through this Bayesian approach, one may
be able to obtain a characterization of the class of optimal MDFs
controlling type I error rates, or when the two types of error
rates are combined, for example, via a weighted linear combination.
The papers \cite{MulParRobRou04,SarZhoGho08,Efr08,EfrAAS08} which
employ Bayes or
empirical Bayes approaches are highly relevant on this front.

Finally, we mention that there are still other aspects of the multiple
decision problem not dealt with in this paper. First is the
extension to situations with composite null and alternative hypotheses.
We indicated some ideas in
Section \ref{subsection-MLR Families} for distributional models
possessing the MLR property, but further and more extensive studies are needed.
Second are possible dependencies among the components in $(X_m, m \in
\mathcal{M}_0(Q))$.
We have assumed that this is an independent collection,
but it is certainly of theoretical and applied relevance to examine
dependent settings.
Potential results in such scenarios will extend those in
\cite{Sar98,BenYek01,SarAOS08}.
In these composite hypotheses and dependent data settings,
we expect that resampling-based ideas and approaches, such as those
in \cite{WesYou93,WesTro08}, will be central.

\section*{Acknowledgments}

The first author is grateful to Dr. James Berger for facilitating his
sabbatical leave visit at the Statistical and Applied Mathematical
Sciences Institute (SAMSI)
during Fall 2008 as this afforded him quality time for generating ideas
relevant to this project.
As such this work was partially supported by the National
Science Foundation (NSF) under Grant DMS-0635449 to SAMSI.
However, any opinions, findings and
conclusions or recommendations expressed in this paper are those of
the authors and do not necessarily reflect the views of the National
Science Foundation.
He is also grateful to Prof. Odd Aalen and Prof. Bo Lindqvist
for facilitating his visits to the University of Oslo and
the Norwegian University of Science and Technology (NTNU) which led to
critical ideas for this project.
The authors are highly grateful to the two reviewers, Associate Editor
and the Editors
for their comments, suggestions and criticisms.
Special thanks to Prof. Sanat Sarkar and Prof. Lan Wang
for a careful reading of an earlier version of the manuscript, and thank
the following for comments or for pointing out references:
Prof. J. Lynch, Dr. A. McLain, Prof. G. Rempala, Prof. J. Sethuraman,
Prof. G.~Taraldsen, Prof. A. Vidyashankar, Prof. L. Wasserman
and Prof. P. Westfall.
We also thank Dr. M. Pe\~na for discussions about microarrays.

\begin{supplement}%[id=suppA]
\stitle{Supplement to ``Power-Enhanced Multiple Decision Functions
Controlling~Family-Wise
Error and False Discovery Rates''}
\slink[doi,text={10.1214/10-AOS844SUPP}]{10.1214/10-AOS844SUPP}
\sdatatype{.pdf}
\sfilename{AoS844Supp.pdf}
\sdescription{The proofs of lemmas, propositions, theorems and
corollaries are provided in this
supplemental article \cite{PenHabWu10Supp}.}
\end{supplement}

%suskaldyti doi

%
\printaddresses

\end{document}